\documentclass[12pt]{article}

\usepackage{amsmath,amssymb,amsthm,mathtools,mathrsfs}
\usepackage{graphicx}
\usepackage{xcolor}
\usepackage{tikz}
\usepackage{newclude}
\usepackage{imakeidx}
\usepackage{bm}

\usepackage[colorlinks=true,
            linkcolor=blue,
            citecolor=blue,
            urlcolor=red,
            bookmarksopen=true,
            pdfborder={0 0 0}
           ]{hyperref}
\makeindex

\allowdisplaybreaks

 \numberwithin{equation}{section}

\newtheorem{thm}{Theorem}[section]

\newtheorem{lema}{Lemma}[section]
\newtheorem{prop}{Proposition}[section]

\newtheorem{rmk}{Remark}[section]

\newcommand{\p}{\partial}

\renewcommand{\Pi}{\mathcal{P}}

\renewcommand{\Xi}{\mathcal{X}}

\let\div\relax
\DeclareMathOperator*{\div}{div}

\DeclarePairedDelimiter\abs{\lvert}{\rvert}
\DeclarePairedDelimiter\norm{\lVert}{\rVert}
\makeatletter
\let\oldabs\abs
\def\abs{\@ifstar{\oldabs}{\oldabs*}}
\let\oldnorm\norm
\def\norm{\@ifstar{\oldnorm}{\oldnorm*}}
\makeatother

\usepackage{scalerel}
\usepackage[usestackEOL]{stackengine}
\def\dashint{\,\ThisStyle{\ensurestackMath{
\stackinset{c}{.2\LMpt}{c}{.5\LMpt}{\SavedStyle-}{\SavedStyle\phantom{\int}}
}
\setbox0=\hbox{$\SavedStyle\int\,$}\kern-\wd0}\int}

\usepackage{adjustbox}
\usepackage{wrapfig}
\usepackage{enumitem}
\usepackage{protosem}

\title{Global strong solutions and asymptotic behavior for arbitrarily large initial data of the 2D compressible Navier-Stokes equations with transport entropy}
\date{}
\author{
\bf\large  Jie Fan$^{a}$, Xiangdi Huang$^{a,*}$\\
\small a. State Key Laboratory of Mathematical Sciences, Academy of Mathematics and Systems Science\\
\small Chinese Academy of Sciences, Beijing 100910, P.R.China;\\
\footnote{*Email addresses: fj0828@outlook.com\,\, (J. Fan); xdhuang@amss.ac.cn\,\, (X. Huang);   } }
\begin{document}
\maketitle

\begin{abstract}
In 1995, Kazhikhov and Vaigant introduced a particular class of isentropic compressible Navier-Stokes equations with variable viscosity coefficients and, for the first time, established the existence of global smooth solutions for arbitrarily large initial data in bounded two-dimensional domains. This result was subsequently extended and refined to accommodate more general constraints on the viscosity coefficients. However, because the proofs in this line of work [Jiu-Wang-Xin, J. Math. Fluid Mech 16 :483–521, 2014; Huang-Li, J. Math. Pures Appl 106:123–154, 2016; Huang-Li, SIAM J. Math. Anal 54 (3):3192–3214, 2022; Fan-Li-Li,  Arch. Ration. Mech. Anal 245 (1):239–278, 2022] rely heavily on the structure of the isentropic equations, they could not be generalized to the broader setting of multidimensional compressible heat-conductive Navier-Stokes–Fourier systems. In this paper, we consider a special class of non-isentropic compressible fluids governed by the two-dimensional compressible Navier-Stokes equations with variable entropy. In this system, the pressure depends nonlinearly on both density and entropy, and the entropy evolves solely through a transport equation-a feature that distinguishes it from the standard Navier-Stokes–Fourier model. We establish, for the first time, the global existence of strong solutions for arbitrarily large initial data on both two-dimensional periodic domains and bounded domains endowed with Navier-slip boundary conditions. For the bounded-domain case, a key step in our analysis is the derivation of new commutator estimates compatible with the slip condition. Our results hold even when the initial density may contain vacuum and require no smallness assumption on the initial data, provided the shear viscosity is constant and the bulk viscosity follows a power-law form $\lambda(\rho)=\rho^\beta$ with $\beta > 4/3$. Moreover, we demonstrate that the density remains uniformly bounded for all time. Consequently, the solution converges to an equilibrium state as time tends to infinity

. \\[4mm]
{\bf Keywords:} global strong solutions; entropy transport; large initial data;\\[4mm]
{\bf Mathematics Subject Classifications (2010):} 76N10, 35B45.\\[4mm]
\end{abstract}
\tableofcontents
\newpage

\section{Introduction}
In this paper, we study the compressible Navier–Stokes equations with entropy transport. This model describes flows of compressible viscous fluids with variable entropy and arises naturally in meteorological applications as a limit system of the full Navier–Stokes–Fourier equations when thermal conduction is neglected and viscous heating is omitted. Such a reduced system is particularly relevant in atmospheric flow modeling and has been extensively studied in the meteorological literature; see \cite{R. Klein}. The governing system consists of the compressible Navier–Stokes equations coupled with an additional transport equation for the entropy, and reads as follows:

\begin{align}
\begin{cases}\label{trans201}
&\partial_t \rho + \operatorname{div} (\rho \mathbf{u}) = 0, \\[2mm]
&\partial_t (\rho \mathbf{u}) + \operatorname{div} (\rho \mathbf{u} \otimes \mathbf{u}) 
- \mu \Delta \mathbf{u} - \nabla((\mu+\lambda)\text{div}u)
+ \nabla P(\rho, s) = 0, \\[2mm]
&\partial_t (\rho s) + \operatorname{div} (\rho s \mathbf{u}) = 0, \\[2mm]
&(\rho, \mathbf{u}, s)|_{t=0} = (\rho_0, \mathbf{u}_0, s_0).
\end{cases}
\end{align}
where $\rho=\rho(x,t),$ and $\mathbf{u}=(u_1(x,t),u_2(x,t)),$ and $s=s(x,t)$ represent the velocity field and the entropy, respectively. The pressure state equation is given by
\begin{equation}\label{pressurestate}
P(\rho, s) = \rho^\gamma \mathcal{T}(s), \quad\gamma > 1, 
\end{equation}
where $\mathcal{T}(\cdot):\mathbb{R}+ \to \mathbb{R}+$ is a smooth, strictly monotone function, positive for all $s > 0$.

Regarding the viscosity coefficients, the shear viscosity \(\mu\) is taken as a positive constant, while the bulk viscosity \(\lambda\) is assumed to depend on the density as a power law:
\begin{equation}
\mu = \text{const} > 0, \quad \lambda(\rho) = b \rho^{\beta}, \quad b > 0, \quad \beta > 0. \end{equation}
For simplicity and without loss of generality, we normalize \(a = b = 1\) throughout the paper.

We consider two types of domains:
\begin{itemize}
    \item Let $\Omega \subset \mathbb{R}^{2}$ be a bounded, simply connected $C^{2,1}$ domain.  
We impose the Navier-slip boundary conditions:
\begin{equation}
\mathbf{u} \cdot n = 0, \qquad \operatorname{rot} \mathbf{u} = -K \mathbf{u} \cdot n^{\perp} \quad \text{on } \partial\Omega,
\end{equation}
where $K \geq 0$ is a smooth function defined on $\partial\Omega$, $n = (n_1, n_2)$ denotes the unit outer normal vector, and $n^{\perp} = (n_2, -n_1)$ is the corresponding unit tangent vector.
    \item the periodic domain: $\mathbb{T}^2= (0,1) \times (0,1)$. 
\end{itemize}

 The mathematical analysis of stability for Navier-Stokes-type systems with transport mechanisms originated with the pioneering work of Lions \cite{plions}, who established conditional stability in three dimensions under the technical assumption that $\gamma \geq 9/5$. 
 In the case of $P(\rho,\bar{s}) = (\rho\bar{s})^\gamma$ with $\gamma > \frac{3}{2}$, Mich\'{a}lek \cite{michalek2015} studied the stability of weak solutions to \eqref{trans201}. Maltese et al. \cite{maltese2016} later extended these stability results to more general pressure laws. Contemporary research has produced significant refinements in solution concepts. 
 Luk\'{a}\v{c}ov\'{a}-Medvid'ov\'{a} and Sch\"{o}mer \cite{lukacova2022} introduced the concept of dissipative solutions for the system in bounded domains $\Omega \subset \mathbb{R}^d$ ($d=2,3$), proving their existence for the full range $\gamma \geq 1$. This represents a considerable improvement over earlier $\gamma$-dependent constraints. For the pressure law $P(\rho,\bar{s}) = (\rho\bar{s})^\gamma$ in the whole space, Zhai, Li, and Zhou \cite{zhai2023} proved the global well-posedness of strong solutions for $\gamma > 1$.

For technical reasons, we begin by proving the corresponding results for the system derived from \eqref{trans2} through the variable transformation $Z \equiv \rho \mathcal{T}^{1/\gamma}(s) = P^{1/\gamma}$:
\begin{align}
\begin{cases}\label{trans1}
&\partial_t \rho + \operatorname{div} (\rho \mathbf{u}) = 0, \\[2mm]
&\partial_t (\rho \mathbf{u}) + \operatorname{div} (\rho \mathbf{u} \otimes \mathbf{u}) 
- \mu \Delta \mathbf{u} - \nabla((\mu+\lambda)\text{div}u)
+ \nabla P(Z) = 0, \\[2mm]
&\partial_t Z + \operatorname{div} (Z \mathbf{u}) = 0, \\[2mm]
&(\rho, \mathbf{u}, Z)|_{t=0} = (\rho_0, \mathbf{u}_0, Z_0).
\end{cases}
\end{align}
and pressure now reads
\begin{equation}
P(Z)=a Z^{\gamma},\,\,a>0,\,\,\gamma>1,
\end{equation}
Without loss of generality, we take $a=1.$

	\begin{thm}\label{RELXMHD221}
Suppose the initial data $(\rho_0,m_0,Z_0)$ satisfy some $q > 2$
		\begin{equation}\label{ia}
		0\leq c_{*}\rho_0\leq Z_0\leq c^{*}\rho_0,\,\,\mathbf{u}_0\in H^1,\,\,m_0=\rho_0 u_0,\, \, (\rho_0,Z_0)\in W^{1,q}.
		\end{equation}
and $(\rho,\mathbf{u}, Z)$ is the unique strong solution presented in Theorem \ref{th1} to system \eqref{trans1}.
 If the maximal existence time $T^{\ast}$ is finite, then there holds
\begin{equation}
\begin{aligned}\label{88888803}
\limsup_{t\rightarrow T^*}\Big(\|\rho\|_{L^{\infty}(0,t;L^{\infty}(\Omega))}+\|Z\|_{L^{\infty}(0,t;L^{\infty}(\Omega))}\Big)= \infty,
\end{aligned}
\end{equation}
\end{thm}
With the blow-up criterion established, we now present our second theorem, which addresses the global existence and large-time behavior of strong solutions.
\begin{thm}\label{th1}
		Let $\Omega \subset \mathbb{R}^2$ be either a simply connected bounded domain with $C^{2,1}$ boundary $\partial\Omega$, or $\Omega = \mathbb{T}^2 = [0,1] \times [0,1]$ be the 2D periodic torus. Let $0 < c_* \leq c^* < \infty$ be constants, and assume  that
\begin{equation}\label{beta43}
\beta>\frac{4}{3},\,\,\gamma>1.
\end{equation}
    Suppose the initial data satisfy \eqref{ia}
        and then there exists a unique global strong solution $(\rho,\mathbf{u},Z)$ to system \eqref{trans1} in $\Omega\times(0,\infty)$, which satisfies:
		\begin{equation}\label{theoremregularity}
			\left\{ \begin{array}{l}
				(\rho,Z)\in C([0,T];W^{1,q}), \ (\rho_t,Z_t)\in L^{\infty}((0,\infty);L^2),\\
				 \mathbf{u} \in L^{\infty}((0,T);H^1)\cap L^{\frac{q+1}{q}}((0,T);W^{2,q}), \\
				t^{\frac{1}{2}}\mathbf{u}\in L^2((0,T);W^{2,q}), t^{\frac{1}{2}}\mathbf{u}_t
				\in L^2((0,T);H^1),\\
                \rho \mathbf{u}\in C([0,T];L^2),\,\sqrt{\rho}\mathbf{u}_t\in L^2(\mathbb{T}^2\times(0,T)),
			\end{array} \right.
		\end{equation}
Moreover, if 
\begin{equation}\label{beta23}
\beta>\frac{3}{2}, 
\end{equation}
    there exists a constant $C > 0$, depending only on $\beta, \gamma, \mu, \|\rho_0\|_{L^\infty}, \|Z_0\|_{L^\infty}$ and $\|\mathbf{u}_0\|_{H^1}$, such that
    \begin{equation}
    \sup_{0 \le t < \infty} (\|\rho(\cdot, t)\|_{L^{\infty}}+ \|Z(\cdot, t)\|_{L^{\infty}})\le C, 
    \end{equation}
    and the following large-time behavior holds:
    \begin{equation}\label{eq:long-time}
    \lim_{t\to\infty}(\|\rho-\overline{\rho_0}\|_{L^p}+\|Z-\overline{Z_0}\|_{L^p}+\|\nabla \mathbf{u}\|_{L^p})=0,
    \end{equation}
    for any $p\in[1,\infty).$
	\end{thm}
\begin{rmk}
Our result appears to provide the first discussion on the global strong solutions for the non-isentropic compressible Navier–Stokes equations with large initial data.
\end{rmk}
\begin{rmk}
In contrast to the result \cite{huanglijmpa}, our analysis removes the restriction $1<\beta<4\gamma-3$ in the large-time behavior.
\end{rmk}
\begin{rmk}
The key objective of this paper is to establish the uniform control between $Z$ and $\rho$:
\begin{equation}
\begin{aligned}\label{equi}
c_{*}\rho \leq Z \leq c^{*}\rho \qquad \text{a.e. in } (0,T)\times\Omega. 
\end{aligned}
\end{equation}
For smooth solutions with strictly positive density, the entropy equation $\eqref{trans1}_3$ can be transformed, via the continuity equation $\eqref{trans1}_1$, into the pure transport equation:
\begin{align}
\begin{cases}\label{trans21}
&\partial_t \rho + \operatorname{div} (\rho \mathbf{u}) = 0, \\[2mm]
&\partial_t (\rho \mathbf{u}) + \operatorname{div} (\rho \mathbf{u} \otimes \mathbf{u}) 
- \mu \Delta \mathbf{u} - \nabla((\mu+\lambda)\text{div}u)
+ \nabla P(\rho, s) = 0, \\[2mm]
&\partial_t  s + u\cdot{\nabla s} = 0, \\[2mm]
&(\rho, \mathbf{u}, s)|_{t=0} = (\rho_0, \mathbf{u}_0, s_0).
\end{cases}
\end{align}
For subsequent convenience, we define
\begin{equation}
\begin{aligned}\label{thetas}
\theta \equiv \mathcal{T}(s) ^{\frac{1}{\gamma}}.
\end{aligned}
\end{equation}
Provided that $\rho_0>0$, we can define $\theta_0 = Z_0/\rho_0$ through condition \eqref{ia}. From \eqref{thetas} and $\eqref{trans21}_3$ it follows that $\theta$ satisfies the same transport equation:
\begin{equation}
\begin{aligned}\label{theta}
\partial_t  \theta + u\cdot{\nabla \theta} = 0,
\end{aligned}
\end{equation}
and  it inherits the uniform bounds $c_* \leq \theta \leq c^*$ for all time. Combining the state equation $Z = \rho \mathcal{T}^{1/\gamma}(s) = \rho \theta$ with the uniform bounds on $\theta$, inequality \eqref{equi} follows directly.

When vacuum occurs, condition \eqref{ia}  gives $Z_0 = 0$ whenever $\rho_0 = 0$. 
Hence $\theta$ cannot be defined directly through $\frac{Z_0}{\rho_0}$ in such regions. 
In the subsequent approximation procedure, constructing a well‑behaved approximation for $\theta$ would otherwise fail. 
To resolve this difficulty, we introduce an auxiliary function $A(t,x)$. 
Its initial value is taken as $A_0 =\theta_0= [\mathcal{T}(s_0)]^{\frac{1}{\gamma}}$, which is well-defined on all of $\Omega$ because $s_0$ is prescribed globally. 
We require $A$ to satisfy the same transport law
\begin{equation}
\begin{aligned}\label{A}
\partial_t A + \mathbf{u}\cdot\nabla A = 0, \qquad A(0,\cdot)=A_0. 
\end{aligned}
\end{equation}
Thus, $A$ is also transported along the particle trajectories.

In the approximation step (see \eqref{appro}), this assignment supplies a reasonable value for $\theta$ precisely where $\rho_0=Z_0=0$. 
Hence we define $\theta$ by
\begin{align}
\theta_0 = [\mathcal{T}(s_0)]^{\frac{1}{\gamma}}=
\begin{cases}
\displaystyle \frac{Z_0}{\rho_0}, & \rho_0>0, \\[8pt]
A_0, & \rho_0=0,
\end{cases}
\end{align}
and $\theta$ satisfies the transport equation \eqref{theta}. 
Because $A$ (and hence $\theta$ where $\rho=0$) inherits the bounds $c_* \leq A_0 \leq c^*$, we obtain globally
\begin{equation}
\begin{aligned}
c_* \leq \theta \leq c^*. 
\end{aligned}
\end{equation}

\end{rmk}

The proof of Theorem~\ref{RELXMHD221} follows a strategy similar to that used for the isentropic compressible Navier-Stokes equations; below, we only outline the main ideas of the proof.
\noindent\textbf{1. Blow-up Criterion.} 
We now sketch the main ideas behind the proofs of Theorems~\ref{RELXMHD221} and~\ref{th1}.
From the local existence theory, there exists a continuous decreasing function $\mathcal{I}:[0,\infty) \to (0,\infty)$ such that
\begin{equation}
T^* \geq \mathcal{I}(D_0), \quad D_0 = \max\left\{ \| (\varrho_0, Z_0) \|_{W^{1,q}}, u_0\in H^1\right\}.
\end{equation}
Thus, if $T^* < \infty$, then necessarily
\begin{equation}\label{kq1}
\limsup_{t \uparrow T^*} \left( \| ( \varrho, Z)(t) \|_{W^{1,q}} +\|\mathbf{u}\|_{H^1}+ \|\rho\|_{L^{\infty}} + \| Z \|_{L^{\infty}} \right) = \infty.
\end{equation}
Assume $T^* < \infty$ but 
\begin{equation}\label{kq}
\sup_{t \in [0, T^*)} \left( \| \rho \|_{L^{\infty}} + \| Z \|_{L^{\infty}} \right) \leq K,
\end{equation}
for some $K > 0$. We will show that \eqref{kq} implies
\begin{equation}
\sup_{t \in [0, T^*)} \left( \| (\rho, Z)(t) \|_{W^{1,q}}+\|u\|_{H^1}+ \|\rho \|_{L^{\infty}} + \| Z \|_{L^{\infty}} \right) \leq K',
\end{equation}
for some $K' > 0$, contradicting \eqref{kq1}. Therefore, we are able to prove \eqref{88888803}.

\vspace{1em}
\noindent\textbf{2. Main Strategy of the Proof Theorem \ref{th1}}
\noindent
\textbf{Proof of Theorem \ref{th1}.}

With the blow-up criterion established, we now proceed to prove Theorem \ref{th1}. The main strategy of the proof can be summarized as follows.

The primary goal is to derive upper bounds for the density $\rho$ and $Z$. We first establish an equivalence between $\rho$ and $Z$, which reduces the problem to obtaining an upper bound solely for the density.
The proof distinguishes two cases according to the range of parameters $\beta$ and $\gamma$.

\noindent\textbf{2.1 Time-independent upper bound for $\rho$ (for $\beta>\frac{3}{2}$).}

Under these stronger conditions, we obtain $\sup_{t\ge 0}\|\rho\|_{L^\infty} \le C$, which is essential for large-time behavior analysis. The strategy differs from the previous case:
\begin{itemize}
    \item Following the idea in \cite{fanliwang}, we make full use of the damping effect of the pressure. By establishing a key inequality 
\begin{equation}
\begin{aligned}
\frac{d}{dt}\int_{\mathbb{T}^2}\rho f^{\alpha}\,dx
   +\int_{\mathbb{T}^2}\rho(f+M)^{\frac{\gamma}{\beta}}f^{\alpha-1}\,dx\leq \text{(other terms)}
\end{aligned}
\end{equation}
and introducing a truncation technique, we ensure that $(f+M)^{\frac{\beta}{\gamma}}$ can be controlled by $\rho^{\gamma}$ together with regular terms. The underlying idea is the following: after a suitable truncation, $|\log\rho|$ is dominated by regular terms when the density is small, whereas it is controlled by $\rho^{\beta}$ when the density is large. Consequently, we obtain
\begin{equation}
(f+M)^{\frac{\beta}{\gamma}} \le C\rho^{\gamma} + \text{(regular terms)} \le CP + \text{(regular terms)},
\end{equation}
where $P$ denotes the pressure term. Exploiting the damping property of the pressure term $P$ to absorb the remaining terms, we finally achieve a time-independent $L^p$ estimate for the density.
\item The key step in our proof lies in establishing a uniform upper bound for the density. We first show, by combining energy-type estimates with the $L^p$-bounds of the density, that the logarithmic term $\log(1+\|\nabla \mathbf{u}\|_{L^2})$ can be controlled by a polynomial function of the $L^\infty$-norm of the density, and further reduce the power dependence of this logarithmic term on the density upper bound (see \eqref{logestimate}). Building upon this, we employ the $W^{1,p}$-estimate for the commutator due to Coifman--Meyer to derive that the $L^\infty$-norm of the commutator term $F$ depends in a controllable manner on the $L^\infty$-norm of the density. Finally, via a Zlotnik inequality, we obtain a global upper bound for the density that is independent of the lower bound of the initial density, thereby completing the proof.
\end{itemize}
\noindent\textbf{2.2 Time-dependent upper bound for $\rho$ (for $\beta > 4/3$).}

Under the conditions $\beta > 4/3$ and $\gamma > 1$, we establish the time-dependent upper bound $\sup_{0 \le t \le T} \|\rho\|_{L^\infty} \le C(T)$. A key step in this analysis is to obtain the estimate $\sup_{0 \le t \le T} \int_{\mathbb{T}^2} \rho|\mathbf{u}|^{2+\nu}\,dx$, from which we derive control of $\|\rho \mathbf{u}\|_{L^q}$ for sufficiently large $q$. This $L^q$-bound of the momentum is essential for estimating the commutator term $F$. When the constants are allowed to depend on time $T$, the resulting estimate for $\|\rho \mathbf{u}\|_{L^q}$ can be expressed with a lower exponent of $R_T$ compared to the time-independent case. This relaxation explains why the time-dependent theory only requires $\beta > 4/3$, whereas a stronger condition $\beta > 3/2$ is needed for the time-independent result.

\noindent\textbf{2.3 Higher-order estimates and global existence.}

Under the given initial condition $m_0 = \rho_0 \mathbf{u}_0$ (see \eqref{ia}), the overall procedure for the high-order estimates proceeds as follows: First, by using Hoff-type time-weighted estimates for the momentum equation $\eqref{trans1}_2$, the weighted energy estimate of the material derivative $\dot{\mathbf{u}}$ is established (see Lemma \ref{timehighorderestimate}), preparing the groundwork for subsequent analysis. Then, by combining the Beale--Kato--Majda inequality with a logarithmic Gronwall-type inequality for the gradients of density and $Z$, a uniform bound on $\|\nabla \rho\|_{L^q}, \|\nabla Z\|_{L^q}$ is obtained. In this process, elliptic regularity theory is applied to translate the higher-order derivative estimates into control on the material derivative and the density gradient. Through successive steps using prior bounds and energy estimates, uniform boundedness $\|\nabla^2\mathbf{u}\|_{L^q}$ is finally achieved, thereby completing the high-order regularity estimates for the global strong solution.

\noindent\textbf{2.4 Vacuum approximation.}

 To complete the proof of Theorem \ref{th1}, it is necessary to construct a vacuum-approximation solution. It is crucial to emphasize the equivalence between $\rho$ and $Z$ in our proof. The key assumption in establishing the equivalence estimates in Section \ref{sec-4} is that $\theta_0$ has uniform lower and upper bounds. Consequently, during the subsequent vacuum approximation procedure, the approximate solution $Z_{\delta,\eta}/\rho_{\delta,\eta}$ must also remain within the interval $[c_*,c^*]$, because the estimates derived in Section \ref{sec-4} depend on the constants $c_*$ and $c^*$. This requires us to choose an appropriate approximation such that $c_* \leq \frac{Z_{\delta,\eta}}{\rho_{\delta,\eta}} \leq c^*$ holds.

The second difficulty arises from the fact that $\frac{Z}{\rho}$ becomes undefined in vacuum regions. To ensure that $\theta_{\delta,\eta}$ is well-defined even in vacuum, and to maintain consistency of the limiting function $\theta$ across both vacuum and non-vacuum regions, we define $\theta_0 = A_0$ when $\rho_0 = 0$, where $A_0$ is given by equation \eqref{A}.

\vspace{1em}
\noindent\textbf{3. Commutator Estimates}

The proof hinges on a systematic reduction of the problem to elliptic estimates, leveraging both the two-dimensional topology and the specific boundary conditions imposed on the velocity field.  

The first step consists in introducing two auxiliary harmonic functions $f_1$ and $f_2$, defined through Neumann problems:

\begin{equation}
\begin{aligned}
\begin{cases}
\Delta f_1 = \mathrm{div}\,\mathrm{div} (\rho \mathbf{u} \otimes \mathbf{u}) &\quad \text{in } \Omega,\\[4pt]
\displaystyle \frac{\p f_1}{\p n} = k\rho|\mathbf{u}|^{2},&\quad \text{on } \partial\Omega,
\end{cases}
\qquad
\begin{cases}
\Delta f_{2} = \mathrm{div}(\rho \mathbf{u}) &\quad \text{in } \Omega,\\[4pt]
\displaystyle \frac{\p f_{2}}{\p n}_{\p\Omega} = 0,&\quad \text{on } \partial\Omega.
\end{cases}
\end{aligned}
\end{equation}

These Neumann conditions are compatible with the \textbf{Navier-slip boundary condition} $\mathbf{u} \cdot n = 0$, which ensures that the boundary terms arising in the subsequent analysis are well-defined.

In a simply connected two-dimensional domain, a divergence-free vector field can be expressed as the rotated gradient of a scalar function:

\begin{equation}
\begin{aligned}
\nabla f_1 - \operatorname{div}(\rho \mathbf{u} \otimes \mathbf{u}) &= \nabla^\perp \psi_1, \\
\nabla f_2 - \rho \mathbf{u} &= \nabla^\perp \psi_2 .
\end{aligned}
\end{equation}

The Navier-slip condition implies $\nabla^\perp \psi_j \cdot n = 0$ on $\partial\Omega$, allowing us to impose homogeneous Dirichlet conditions $\psi_j = 0$ on $\partial\Omega$. Consequently, $\psi_1$ and $\psi_2$ solve Poisson equations:

\begin{equation}
\begin{aligned}
\Delta \psi_1 &= -\nabla^\perp \cdot \operatorname{div}(\rho \mathbf{u} \otimes \mathbf{u}), \\
\Delta \psi_2 &= -\nabla^\perp \cdot (\rho \mathbf{u}) .
\end{aligned}
\end{equation}

These are standard elliptic problems. Applying $W^{2,p}$ regularity theory for the Laplace operator with Dirichlet boundary conditions yields:

\begin{equation}
\begin{aligned}
\|\nabla \psi_1\|_p &\lesssim \|\rho \mathbf{u} \otimes \mathbf{u}\|_p, \\
\|\nabla \psi_2\|_p &\lesssim \|\rho \mathbf{u}\|_p .
\end{aligned}
\end{equation}

Expanding the right-hand sides using the product rule and applying H\"older's inequality separates the terms into products of $\|\rho \mathbf{u}\|_{p_1}$ and $\|\nabla \mathbf{u}\|_{p_2}$. Finally, returning to the definition of $F = f_1 - \mathbf{u} \cdot \nabla f_2$ and expressing $\nabla f_1$ and $\nabla f_2$ through $\psi_1$ and $\psi_2$, we obtain the desired estimate for $\|\nabla F\|_p$.

This proof combines geometric properties of two-dimensional domains with the structural implications of Navier-slip boundary conditions to transform a nonlinear commutator into a system of elliptic boundary value problems, providing a robust framework applicable to broader contexts in fluid regularity theory.

The rest of the paper is organized as follows.
In Section \ref{sec-2}, we present a series of auxiliary lemmas. The proof of Theorem \ref{th1} in periodic domains is carried out in Section \ref{sec-4}. Section \ref{sec-5} establishes the validity of Theorem \ref{th1} for bounded domains; here, the detailed analysis is provided specifically for the commutator estimates in the bounded case.

\section{Preliminary}
\label{sec-2}

The following well-known local existence theory, where the initial density is strictly away from vacuum, can be found in \cite{lukacova2025}.

\begin{lema} \label{lem:local-existence}
    Assume that $(\rho_0, m_0, Z_0)$ satisfies
    \begin{equation} \label{eq:initial-regularity}
        (Z_0, \rho_0) \in H^2, \quad \mathbf{u}_0 \in H^2, \quad \inf_{x \in \Omega} (\rho_0(x), Z_0(x)) > 0, \quad m_0 = \rho_0 u_0.
    \end{equation}
    Then there are a small time $T > 0$ and a constant $C_0, C_1 > 0$ both depending only on $\|Z_0\|_{H^2}$, $\|\rho_0\|_{H^2}$, $\|u_0\|_{H^2}$, and $\inf \rho_0(x)$ such that there exists a unique strong solution $(\rho, \mathbf{u}, Z)$ to the problem \eqref{trans1} in $\Omega \times (0, T)$ satisfying
    \begin{equation} \label{eq:local-regularity}
        \begin{cases}
            (\rho,Z) \in C([0, T]; H^2), & (\rho_t,Z_t) \in C([0, T]; H^1), \\
            \mathbf{u} \in L^2(0, T; H^3), & \mathbf{u}_t \in L^2(0, T; H^1), \\
            \mathbf{u}_t \in L^2(0, T; H^2), & \mathbf{u}_{tt} \in L^2((0, T) \times \Omega),
        \end{cases}
    \end{equation}
    and
    \begin{equation} \label{eq:lower-bound}
\inf_{(x,t) \in \Omega \times (0,T)} \rho(x,t) \geq C_0> 0, \quad
\inf_{(x,t) \in \Omega \times (0,T)} Z(x,t) \geq C_1 > 0.
\end{equation}
\end{lema}

The following Poincar\'{e}--Sobolev and Brezis--Wainger inequalities will be used frequently.

\begin{lema} \label{lem:poincare-sobolev}
    (See \cite{H. Engler,H. Brezis,O.A. Ladyzenskaja}.) There exists a positive constant $C$ depending only on $\mathbb{T}^2$ such that every function $\mathbf{u} \in H^1(\mathbb{T}^2)$ satisfies for $2 < p < \infty$,
    \begin{equation} \label{eq:poincare-sobolev}
        \begin{aligned}
            \|\mathbf{u} - \bar{\mathbf{u}}\|_{L^p} &\leq C p^{1/2} \|\mathbf{u} - \bar{\mathbf{u}}\|_{L^2}^{2/p} \|\nabla \mathbf{u}\|_{L^2}^{1-2/p}, \\
            \|\mathbf{u}\|_{L^p} &\leq C p^{1/2} \|\mathbf{u}\|_{L^2}^{2/p} \|u\|_{H^1}^{1-2/p}.
        \end{aligned}
    \end{equation}
    Moreover, for $q > 2$, there exists some positive constant $C$ depending only on $q$ and $\mathbb{T}^2$ such that every function $v \in W^{1,q}(\mathbb{T}^2)$ satisfies
    \begin{equation} \label{eq:brezis-wainger}
        \begin{aligned}
            \|v\|_{L^\infty} \leq C \|\nabla v\|_{L^2} \ln^{1/2}(e + \|\nabla v\|_{L^q}) + C \|v\|_{L^2} + C.
        \end{aligned}
    \end{equation}
\end{lema}

The following Poincar\'{e} type inequality can be found in \cite{Feireisl2004}.

\begin{lema} \label{lem:poincare-type}
    Let $v \in H^1(\mathbb{T}^2)$, and let $\rho$ be a non-negative function such that
    \begin{equation}
        \begin{aligned}
            0 < M_1 \leq \int_{\mathbb{T}^2} \rho \, dx, \quad \int_{\mathbb{T}^2} \rho^\gamma \, dx \leq M_2,
        \end{aligned}
    \end{equation}
    with $\gamma > 1$. Then there is a constant $C$ depending solely on $M_1$, $M_2$, and $\gamma$ such that
    \begin{equation} \label{eq:poincare-type}
        \begin{aligned}
            \|v\|_{L^2(\mathbb{T}^2)} \leq C \int_{\mathbb{T}^2} \rho v^2 \, dx + C \|\nabla v\|_{L^2(\mathbb{T}^2)}.
        \end{aligned}
    \end{equation}
\end{lema}

Then, we state the following Beale--Kato--Majda type inequality which will be used later to estimate $\|\nabla \mathbf{u}\|_{L^\infty}$ and $\|\nabla \rho\|_{L^p}, \|\nabla Z\|_{L^p}$.

\begin{lema} \label{lem:BKM}
    (See \cite{J.T. Beale, T. Kato, huanglixinsiam}.) For $2 < q < \infty$, there is a constant $C(q)$ such that the following estimate holds for all $\nabla \mathbf{u} \in W^{1,q}(\mathbb{T}^2)$,
    \begin{equation} \label{eq:BKM}
        \begin{aligned}
            \|\nabla \mathbf{u}\|_{L^\infty} \leq C \bigl( \|\operatorname{div} \mathbf{u}\|_{L^\infty} + \|\operatorname{rot} \mathbf{u}\|_{L^\infty} \bigr) \log(e + \|\nabla^2 \mathbf{u}\|_{L^q}) + C \|\nabla \mathbf{u}\|_{L^2} + C.
        \end{aligned}
    \end{equation}
\end{lema}

Next, let $\Delta^{-1}$ denote the Laplacian inverse with zero mean on $\mathbb{T}^2$ and $\mathcal{R}_i = (-\triangle)^{-1/2} \partial_i$ be the usual Riesz transform on $\mathbb{T}^2$. Moreover, let $\mathcal{H}^1(\mathbb{T}^2)$ and $\mathcal{BMO}(\mathbb{T}^2)$ stand for the usual Hardy and BMO spaces:
\begin{equation}
    \begin{aligned}
        &\mathcal{H}^1 = \Bigl\{ f \in L^1(\mathbb{T}^2) : \|f\|_{\mathcal{H}^1} = \|f\|_{L^1} + \|\mathcal{R}_1 f\|_{L^1} + \|\mathcal{R}_2 f\|_{L^1} < \infty, \quad \bar{f} = 0 \Bigr\}, \\
        &\mathcal{BMO} = \Bigl\{ f \in L^1_{\text{loc}}(\mathbb{T}^2) : \|f\|_{\mathcal{BMO}} < \infty \Bigr\}
    \end{aligned}
\end{equation}
with
\begin{equation}
    \begin{aligned}
        \|f\|_{\mathcal{BMO}} = \sup_{\substack{x \in \mathbb{T}^2 \\ r \in (0, d)}} \frac{1}{|\Omega_r(x)|} \int_{\Omega_r(x)} \Bigl| f(y) - \frac{1}{|\Omega_r(x)|} \int_{\Omega_r(x)} f(z) \, dz \Bigr| \, dy,
    \end{aligned}
\end{equation}
where $d$ is the diameter of $\mathbb{T}^2$, $\Omega_r(x) = \mathbb{T}^2 \cap B_r(x)$, and $B_r(x)$ is a ball with center $x$ and radius $r$. Consider the composition of two Riesz transforms, $\mathcal{R}_i \circ \mathcal{R}_j$ ($i, j = 1, 2$). There is a representation of this operator as a singular integral
\begin{equation}
    \begin{aligned}
        \mathcal{R}_i \circ \mathcal{R}_j (f)(x) = \text{p.v.} \int K_{ij}(x - y) f(y) \, dy,
    \end{aligned}
\end{equation}
where the kernel $K_{ij}(x)$ ($i, j = 1, 2$) has a singularity of the second order at $0$ and
\begin{equation}
    \begin{aligned}
        |K_{ij}(x)| \leq C |x|^{-2}, \quad x \in \mathbb{T}^2.
    \end{aligned}
\end{equation}
Given a function $b$, define the linear operator
\begin{equation}
    \begin{aligned}
        [b, \mathcal{R}_i \mathcal{R}_j](f) \triangleq b \mathcal{R}_i \circ \mathcal{R}_j (f) - \mathcal{R}_i \circ \mathcal{R}_j (b f), \quad i, j = 1, 2.
    \end{aligned}
\end{equation}
This operator can be written as a convolution with the singular kernel $K_{ij}$,
\begin{equation}
    \begin{aligned}
        [b, \mathcal{R}_i \mathcal{R}_j](f)(x) \triangleq \text{p.v.} \int K_{ij}(x - y) \bigl( b(x) - b(y) \bigr) f(y) \, dy, \quad i, j = 1, 2.
    \end{aligned}
\end{equation}

The following properties of the commutator $[b, \mathcal{R}_i \mathcal{R}_j](f)$ will be useful for our analysis.

\begin{lema} \label{lem:commutator}(See \cite{R. Coifman,R. CoifmanY. Meyer})
    Let $b, f \in C^\infty(\mathbb{T}^2)$. Then for $p \in (1, \infty)$, there is $C(p)$ such that
    \begin{equation} \label{eq:commutator-Lp}
        \begin{aligned}
            \bigl\| [b, \mathcal{R}_i \mathcal{R}_j](f) \bigr\|_{L^p} \leq C(p) \|b\|_{\mathcal{BMO}} \|f\|_{L^p}.
        \end{aligned}
    \end{equation}
    Moreover, for $q_i \in (1, \infty)$ ($i = 1, 2, 3$) with $q_1^{-1} = q_2^{-1} + q_3^{-1}$, there is a positive constant $C$ depending only on $q_i$ ($i = 1, 2, 3$) such that
    \begin{equation} \label{eq:commutator-gradient}
        \begin{aligned}
            \bigl\| \nabla [b, \mathcal{R}_i \mathcal{R}_j](f) \bigr\|_{L^{q_1}} \leq C \|\nabla b\|_{L^{q_2}} \|f\|_{L^{q_3}}.
        \end{aligned}
    \end{equation}
\end{lema}
Finally, the following Zlotnik inequality will be used to get the time-independent upper bound of the density $\rho$.
\begin{lema} \label{lem:Zlotnik}
    (See \cite{A.A. Zlotnik}.) Let the function $y \in W^{1,1}(0, T)$ satisfy
    \begin{equation}
        \begin{aligned}
            y'(t) = g(y) + h'(t) \quad \text{on } [0, T], \quad y(0) = y^0
        \end{aligned}
    \end{equation}
    with $g \in C(\mathbb{R})$ and $h \in W^{1,1}(0, T)$. If $g(\infty) = -\infty$ and
    \begin{equation} \label{eq:Zlotnik-condition}
        \begin{aligned}
            h(t_2) - h(t_1) \leq N_0 + N_1 (t_2 - t_1)
        \end{aligned}
    \end{equation}
    for all $0 \leq t_1 < t_2 \leq T$ with some $N_0 \geq 0$ and $N_1 \geq 0$, then
    \begin{equation}
        \begin{aligned}
            y(t) \leq \max \{ y^0, \bar{\zeta} \} + N_0 < \infty \quad \text{on } [0, T],
        \end{aligned}
    \end{equation}
    where $\bar{\zeta}$ is a constant such that
    \begin{equation}
        \begin{aligned}
            g(\zeta) \leq -N_1 \quad \text{for} \quad \zeta \geq \bar{\zeta}.
        \end{aligned}
    \end{equation}
\end{lema}

\section{Proof of Theorem \ref{th1} on the Periodic Domain}
\label{sec-4}

\subsection{ Equivalence between $Z$ and $\rho$}
Throughout this section, we work under the following framework. The initial data $(\rho_0, Z_0 \mathbf{m}_0)$ satisfies \eqref{eq:initial-regularity}, and $(\rho, \mathbf{u}, Z)$ denotes the unique strong solution to \eqref{trans1} on $\mathbb{T}^2\times(0, T]$ given by Lemma \ref{lem:local-existence}.
Under the conditions of Lemma \ref{lem:local-existence}, system  is equivalent to the following system 
\begin{equation}\label{trans2}
\begin{aligned}
&\begin{cases}
\partial_t \rho + \operatorname{div} (\rho \mathbf{u}) = 0, \\[2mm]
\partial_t (\rho \mathbf{u}) + \operatorname{div} (\rho \mathbf{u} \otimes \mathbf{u}) 
- \mu \Delta \mathbf{u} - (\lambda + \mu) \nabla \operatorname{div} \mathbf{u} 
+ \nabla P(\rho, s) = 0, \\[2mm]
\partial_t (s) + \mathbf{u}\cdot{\nabla s} = 0, \\[2mm]
(\rho, \mathbf{u}, s)|_{t=0} = (\rho_0, \mathbf{u}_0, s_0).
\end{cases}
\end{aligned}
\end{equation}

In this case, the entropy equation is decoupled from the continuity equation. By relation \eqref{thetas}, it follows that the variable $\theta$ satisfies the equation 
\begin{equation}\label{temperature}
\begin{aligned}
\theta_t + \mathbf{u} \cdot \nabla \theta = 0.
\end{aligned}
\end{equation}
Under assumption \eqref{ia}, we have 
\begin{equation}\label{initial temperature}
\begin{aligned}
c_{*}\leq\theta_0\leq c^{*},
\end{aligned}
\end{equation}
By \eqref{initial temperature} and equation \eqref{temperature}, we have
\begin{equation}
\begin{aligned}
c_{*} \leq \theta(t,x) \leq c^{*} \quad \text{for all } (t,x) \in \mathbb{T}^2\times (0,\infty).
\end{aligned}
\end{equation}
From the above conditions, it follows that $Z$ is bounded:
\begin{equation}
\begin{aligned}\label{equzrho}
c_{*} \rho(t,x) \leq Z(t,x) \leq c^{*} \rho(t,x) \quad \text{for all } (t,x)\in\mathbb{T}^2\times (0,\infty).
\end{aligned}
\end{equation}

\begin{equation}\label{3.12}
\begin{aligned}
G\triangleq(2\mu+\lambda(\rho))\text{div}\mathbf{u}-(P(Z)-\overline{P(Z)}),\,\,w\triangleq\nabla^{\perp}\cdot{\mathbf{u}}=\partial_2u_1-\partial_1u_2.
\end{aligned}
\end{equation}
\begin{equation}\label{3.2}
\begin{aligned}
A_1^2(t) \triangleq  \int_{\mathbb{T}^2} \left( \omega(t)^2 + \frac{G(t)^2}{2\mu + \lambda(\rho(t))} \right) dx,
\end{aligned}
\end{equation}

\begin{equation}\label{3.3}
\begin{aligned}
A_2^2(t)\triangleq \int_{\mathbb{T}^2} \rho(t) |\dot{\mathbf{u}}(t)|^2 dx,
\end{aligned}
\end{equation}

\begin{equation}\label{3.4}
\begin{aligned}
A_3^2(t) \triangleq \int_{\mathbb{T}^2} \left( (2\mu + \lambda(\rho(t))) (\mathrm{div} \mathbf{u}(t))^2 + \mu \omega(t)^2 \right) dx,
\end{aligned}
\end{equation}
and
\begin{equation}\label{3.5}
\begin{aligned}
R_T\triangleq \sup_{0\leq t\leq T}
\|\rho(\cdot,t)\|_{L^\infty}\end{aligned}
\end{equation}
\subsection{A priori estimates (I): upper bound of density}
\subsubsection{Time-independent upper bound of the density}

The analysis begins with the derivation of the basic energy estimate.

\begin{lema}\label{lem:energy}
There exists a constant $C>0$, depending only on $c_{*},\gamma$, $\|Z_0\|_{L^\gamma}$, and $\|\sqrt{\rho_0}\mathbf{u}_0\|_{L^2}$, such that satisfies
\begin{equation}
\begin{aligned}\label{basicestimate}
\sup_{0\le t\le T}\int_{\Omega}\big(\rho|\mathbf{u}|^2+\frac{Z^\gamma}{\gamma-1}\big)\,dx
+\int_0^T\!\!\int_{\mathbb{T}^2}\Bigl((2\mu+\lambda(\rho))(\operatorname{div}\mathbf{u})^2+\mu\omega^2\Bigr)\,dx\,dt\le C.
\end{aligned}
\end{equation}

\end{lema}
\begin{proof}
From the $\eqref{trans1}_3$ we obtain the pressure equation
\begin{equation}\label{eq:pressure}
P_t+\operatorname{div}(P\mathbf{u})+(\gamma-1)P\operatorname{div}\mathbf{u}=0.
\end{equation}
Integrating \eqref{eq:pressure} over $\mathbb{T}^2$ gives
\begin{equation}\label{eq:pressure_int}
\frac{d}{dt}\int_{\mathbb{T}^2}\frac{Z^\gamma}{\gamma-1}\,dx+\int_{\mathbb{T}^2}P\operatorname{div}u\,dx=0.
\end{equation}
Multiplying $\eqref{trans1}_1$ by $\mathbf{u}$, integrating over $\mathbb{T}^2$, and combining with \eqref{eq:pressure_int} yields,
\begin{equation}\label{eq:energy_id}
\begin{aligned}
\frac{d}{dt}&\int_{\Omega}\Bigl(\frac12\rho|\mathbf{u}|^2+\frac{Z^\gamma}{\gamma-1}\Bigr)\,dx \\
&+\int_{\Omega}\Bigl((2\mu+\lambda(\rho))(\operatorname{div}\mathbf{u})^2+\mu\omega^2\Bigr)\,dx = 0.
\end{aligned}
\end{equation}
Integrating \eqref{eq:energy_id} over $(0,T)$ yields the desired estimate and completes the proof.
\end{proof}
First, from the continuity equation \eqref{trans1}$_1$ we have
\begin{equation}\nonumber
\begin{aligned}
\frac{d}{dt}\int_{\mathbb{T}^2} \rho(x,t)\,dx = 0,
\end{aligned}
\end{equation}
so the total mass is conserved:
\begin{equation}\nonumber
\begin{aligned}
\int_{\mathbb{T}^2} \rho(x,t)\,dx = \int_{\mathbb{T}^2} \rho_0(x)\,dx = 1.
\end{aligned}
\end{equation}
Hence
\begin{equation}\nonumber
\begin{aligned}
\|\rho(\cdot,t)\|_{L^\infty} \geq \frac{1}{|\mathbb{T}^2|} \int_{\mathbb{T}^2} \rho(x,t)\,dx \geq 1 > 0,
\end{aligned}
\end{equation}
and by the definition of $R_T$ in \eqref{3.5},
\[
R_T \geq \|\rho(\cdot,t)\|_{L^\infty} \geq 1. \tag{3.11}
\]
From the definition of $A_3(t)$ in \eqref{3.4} and the non-negativity of $\rho$, we have
\begin{equation}\label{a3u}
\begin{aligned}
\mu\|\nabla \mathbf{u}\|_{L^{2}}^{2}\leq A_{3}^{2}(t)\leq\left(2\mu+R_{T}^{\beta}\right)\|\nabla u\|_{L^{2}}^{2}, 
\end{aligned}
\end{equation}
where $R_T$ is given by \eqref{3.5}.

\begin{lema}
\label{lemma:velocity_T2}
Let $\Omega = \mathbb{T}^2$.  For any $p \ge 1$, there exists a constant $C = C(p,\gamma) > 0$ such that
\begin{equation}\label{ulp}
\begin{aligned}
\| \mathbf{u} \|_{L^p} 
\le C  \,
  \| \nabla \mathbf{u} \|_{L^2}.
\end{aligned}
\end{equation}
\end{lema}
\begin{proof}
 Integrating the momentum equation $\eqref{trans1}_2$ over $\mathbb{T}^2$, we obtain
    \begin{equation}\label{monu}
\begin{aligned}
        \frac{d}{dt} \int_{\mathbb{T}^2} \rho \mathbf{u} \, dx = 0,
   \end{aligned}
\end{equation}

     To obtain uniform-in-time $L^p$ estimates for the density, we need to control the velocity solely by $\|\nabla \mathbf{u}\|_{L^2}$ without introducing time-dependent constants. Since the standard Poincar\'e inequality in the periodic setting only gives
    \begin{equation}\nonumber
\begin{aligned}
        \|\mathbf{u} - \overline{\mathbf{u}}\|_{L^p} \le C(p) \|\nabla \mathbf{u}\|_{L^2},
    \end{aligned}
\end{equation}
we replace the usual mean $\overline{\mathbf{u}}$ by the conserved momentum average $\overline{\rho \mathbf{u}}$, which is constant in time. Using the triangle inequality,
    \begin{equation}\nonumber
\begin{aligned}
        \| \mathbf{u} \|_{L^p}
        \le \| \mathbf{u} - \overline{\mathbf{u}} \|_{L^p}
          + \| \overline{\mathbf{u}} - \overline{\rho \mathbf{u}} \|_{L^p}.
    \end{aligned}
\end{equation}
    The first term is bounded by $C(p) \|\nabla \mathbf{u}\|_{L^2}$. For the second term, note that
    \begin{equation}\nonumber
\begin{aligned}
        \overline{\mathbf{u}} - \overline{\rho \mathbf{u}}
        = \frac{1}{|\mathbb{T}^2|} \int_{\mathbb{T}^2} (\rho - \overline{\rho}) (\mathbf{u} - \overline{\mathbf{u}}) \, dx,
    \end{aligned}
\end{equation}
    because $\int_{\mathbb{T}^2} (\rho - \overline{\rho}) \overline{\mathbf{u}} \, dx = 0$. Hence,
    \begin{equation}\nonumber
\begin{aligned}
        |\overline{\mathbf{u}} - \overline{\rho \mathbf{u}}|
        \le C\| \rho - \overline{\rho} \|_{L^\gamma}
           \| \mathbf{u} - \overline{\mathbf{u}} \|_{L^{\gamma'}},
    \end{aligned}
\end{equation}
    where $\gamma'$ is the conjugate exponent of $\gamma$. Applying the Poincar\'e inequality again and the energy estimate $\|\mathbf{u} - \overline{\mathbf{u}}\|_{L^{\gamma'}} \le C \|\nabla \mathbf{u}\|_{L^2}$, we obtain
     \begin{equation}\nonumber
\begin{aligned}
        |\overline{\mathbf{u}} - \overline{\rho \mathbf{u}}|
        \le C \|\rho\|_{L^\gamma} \|\nabla \mathbf{u}\|_{L^2}.
    \end{aligned}
\end{equation}
    Combining these estimates yields
    \begin{equation}\nonumber
\begin{aligned}
        \| \mathbf{u} \|_{L^p}
        \le C \big( 1 + \|\rho\|_{L^\gamma} \big) \|\nabla \mathbf{u}\|_{L^2}.
    \end{aligned}
\end{equation}

\end{proof}

Building upon the time-dependent $L^p$ integrability estimate for the density obtained by Vaigant–Kazhikhov \cite{VaigantKazhikhov}, we adopt the strategy of \cite{fanliwang} and fully exploit the damping structure embedded in the equations. This enables us to establish a uniform-in-time higher integrability bound for the density. This uniform-in-time estimate serves as a critical step toward further relaxing the restriction on the parameter $\beta$ and provides a key analytical tool for studying the long-time dynamics of solutions.
\begin{lema}\label{lem:Phi-bound}
Let $g_+(x) \triangleq  \max\{g,0\}$. For any $\alpha \in [3,\infty)$, there exist constants $M, C > 0$ depending only on $c_{*}, c^{*},\mu, \alpha, \beta, \gamma, \rho_0, \mathbf{u}_0$, and the domain $\mathbb{T}^2$, but independent of $T$, such that
\begin{equation}\label{eq:lpdensity}
\sup_{0 \le t \le T} \int_{\mathbb{T}^2} \rho f^{\alpha} \, dx \le C,
\end{equation}
where the function $f$ is defined as
\begin{equation}\label{eq:f-def}
f \triangleq \bigl( \theta(\rho) +\psi - M \bigr)_+.
\end{equation}
\end{lema}

\begin{proof}
Accordingly, we reformulate the momentum equations in the following form:
\begin{equation}\label{momequ1}
\rho\dot{\mathbf{u}}=\nabla G+\mu\nabla^{\perp}\omega.
\end{equation}
Since $\Delta^{-1}$ is well-defined on $\mathbb{T}^2$ for functions with zero average, 
applying $\Delta^{-1}\operatorname{div}$ to \eqref{momequ1} yields
\begin{equation}
\begin{aligned}\label{gchange}
G - \overline{G} = \frac{D}{Dt}\psi - \bigl[u_{i},\mathcal{R}_{ij}\bigr](\rho u_{j}),
\end{aligned}
\end{equation}
where 
$\psi \triangleq\Delta^{-1}\operatorname{div}(\rho \bm{u})$, and the commutator is defined by
\begin{equation}
\begin{aligned}\label{commutator}
\bigl[u_{i},\mathcal{R}_{ij}\bigr](\rho u_{j}) \triangleq
\bm{u} \cdot \nabla \Delta^{-1}\operatorname{div}(\rho \bm{u})-\Delta^{-1}\operatorname{div}\operatorname{div}(\rho \bm{u}\otimes \bm{u}).
\end{aligned}
\end{equation}
Combining $\eqref{trans1}_1$ and \eqref{gchange}, we obtain
\begin{equation}
\begin{aligned}\label{changeequ}
\frac{D}{Dt}\bigl(\theta(\rho)+\psi\bigr)+P
      =\bigl[u_{i},\mathcal{R}_{ij}\bigr](\rho u_{j})+\overline{P}-\overline{G}.
\end{aligned}
\end{equation}
Multiplying \eqref{changeequ} by $\rho f^{\alpha-1}$ and integrating over $\mathbb{T}^2$ leads to
\begin{equation}
\begin{aligned}\label{equ4.1}
\frac{d}{dt}\int_{\mathbb{T}^2}\rho f^{\alpha}\,dx
   +\int_{\mathbb{T}^2}\rho P(Z)f^{\alpha-1}\,dx
   =& \alpha\int_{\mathbb{T}^2}\rho\bigl(\bigl[u_{i},\mathcal{R}_{ij}](\rho u_{j})+\overline{P}-\overline{G}\bigl)\,f^{\alpha-1}\,dx,\\
\end{aligned}
\end{equation}

We now proceed to handle equation \eqref{equ4.1} in a detailed manner. The main steps are outlined below.

\noindent\textbf{Step 1: Dealing with the damping term}

In order to fully utilize the damping term of the equations, the second term on the right-hand side of \eqref{equ4.1} must be processed.
Let 
\begin{equation}
\begin{aligned}\nonumber
        f+M=\theta(\rho)+\psi.
\end{aligned}\end{equation}
We aim to prove:
\begin{equation}
\begin{aligned}\label{final}
       \Big(f+M\Big)^{\frac{\gamma}{\beta}}\leq C\Big(|\psi|^{\frac{\gamma}{\beta}}+P(Z)\Big).
\end{aligned}\end{equation}
\textbf{Case 1: $\rho\leq 1.$} 

Since $\rho\leq 1,$ we have $\log\rho\leq 0$ and 
\begin{equation}
\begin{aligned}\label{fgeq0}
2\mu\log\rho+\frac{1}{\beta}\rho^{\beta}+\psi-M\geq 0.\end{aligned}\end{equation}
Choose $M>10\beta.$ Then under the condition \eqref{fgeq0}, we get
\begin{equation}
\begin{aligned}\label{f1geq0}
\psi\geq (M-\frac{1}{\beta}\rho^{\beta})-2\mu\log\rho\geq -2\mu\log\rho.
\end{aligned}\end{equation}
We can then write
\begin{equation}
\begin{aligned}
f+M= \psi+\theta(\rho)\leq \psi+\frac{1}{\beta}\rho^{\beta}+2\mu|\log\rho|.
\end{aligned}\end{equation}
Using \eqref{f1geq0}, we obtain $$|\log\rho|\leq C\psi.$$
and thus
\begin{equation}
\begin{aligned}\label{rhosmall}
\Big(f+M\Big)^{\frac{\gamma}{\beta}}\leq C\Big(|\psi|^{\frac{\gamma}{\beta}}+\rho^{\gamma}\Big)\leq C\Big(|\psi|^{\frac{\gamma}{\beta}}+P(Z)\Big),
\end{aligned}\end{equation}
where, by \eqref{equzrho}, the constant $C$ depends on $c_{*}.$

\noindent\textbf{Case 2: $\rho> 1$.}

From the positivity condition $\log\rho>0,$ we obtain 
\begin{equation}
\begin{aligned}
\theta(\rho)\geq\frac{1}{\beta}\rho^{\beta}>0.
\end{aligned}\end{equation}
When $\rho > 1$ and $\beta > 1$, the elementary inequality $\log\rho \leq \rho^{\beta}$ holds, which allows us to absorb the logarithmic term into the power term. Consequently, there exists a constant $C > 0$  such that
\begin{equation}
\begin{aligned}\label{rhobig}
\Big(f+M\Big)^{\frac{\gamma}{\beta}}\leq C\Big(|\psi|^{\frac{\gamma}{\beta}}+(\theta(\rho))^{\frac{\gamma}{\beta}}\Big)\leq C\Big(|\psi|^{\frac{\gamma}{\beta}}+\rho^{\gamma}\Big)\leq C\Big(|\psi|^{\frac{\gamma}{\beta}}+P(Z)\Big).
\end{aligned}\end{equation}
Combining estimates \eqref{rhosmall} and \eqref{rhobig} yields \eqref{final}.

\noindent\textbf{Step 2: On the integrability Estimates for the density } 

We first perform a splitting of the domain, then obtain
\begin{equation}
\begin{aligned}
\int_{\mathbb{T}^2} \rho^{\alpha\beta+1} \, dx 
&= \int_{\{\rho > 1\}} \rho^{\alpha\beta+1} \, dx + \int_{\{\rho \leq 1\}} \rho^{\alpha\beta+1} \, dx \\
&\leq \int_{\{\rho > 1\}} \rho \, \big( \theta(\rho) \big)^{\alpha} \, dx + 1,\notag
\end{aligned}
\end{equation}
where we use the fact that $\theta(\rho) \sim \rho^{\beta}$ for $\rho > 1.$ 
Next, recalling the definition $f = (\theta(\rho) + \psi - M)_+$, we can bound $\theta(\rho)$ from above by $f + |\psi| + M$, leading to:
\begin{equation}
\begin{aligned}
\int_{\{\rho > 1\}} \rho \, \big( \theta(\rho) \big)^{\alpha} \, dx 
&\leq \int_{\{\rho > 1\}} \rho \, \big( f + |\psi| + M \big)^{\alpha} \, dx \\
&\leq C \left( \int_{\mathbb{T}^2} \rho f^{\alpha} \, dx + \int_{\mathbb{T}^2} \rho |\psi|^{\alpha} \, dx + M^{\alpha} \int_{\mathbb{T}^2} \rho \, dx \right).\notag
\end{aligned}
\end{equation}
Now we estimate the term $\int_{\mathbb{T}^2} \rho |\psi|^{\alpha} \, dx$. 
Applying the Sobolev embedding theorem together with elliptic estimates yields
\begin{equation}
\|\psi\|_{L^p} \leq C\|\nabla\psi\|_{L^{\frac{2p}{p+2}}}\leq\|\rho u\|_{L^{\frac{2p}{p+2}}}\\
\leq C\|\rho\|^{\frac{1}{2}}_{L^p}\|\sqrt{\rho}u\|_{L^2}\leq C\|\rho\|^{\frac{1}{2}}_{L^p}.
\end{equation} so we obtain
\begin{equation}
\int_{\mathbb{T}^2} \rho |\psi|^{\alpha} \, dx \leq C\|\rho\|_{L^{\alpha\beta+1}} \, \|\psi\|_{L^{(\alpha\beta+1)/\beta}}^{\alpha} \leq C \|\rho\|_{L^{\alpha\beta+1}}^{\frac{\alpha}{2} + 1}.
\end{equation}
Combining the estimates above, we arrive at:
\begin{equation}
\begin{aligned}
\int_{\mathbb{T}^2} \rho^{\alpha\beta+1} \, dx 
&\leq C \left( \int_{\mathbb{T}^2} \rho f^{\alpha} \, dx + \|\rho\|_{L^{\alpha\beta+1}}^{\frac{\alpha}{2} + 1} + M^{\alpha} \right).
\end{aligned}
\end{equation}
Finally, since $\frac{\alpha}{2} + 1 < \alpha\beta + 1$ for $\beta > 1$, we can apply Young’s inequality to absorb the term $\|\rho\|_{L^{\alpha\beta+1}}^{\frac{\alpha}{2} + 1}$ into the left-hand side, yielding
\begin{equation}\label{rhof}
\int_{\mathbb{T}^2} \rho^{\alpha\beta+1} \, dx \leq C \left( \int_{\mathbb{T}^2} \rho f^{\alpha} \, dx + M^{\alpha} \right).
\end{equation}
By inserting equation \eqref{final} into equation \eqref{equ4.1}, we find that
\begin{equation}
\begin{aligned}\label{final1}
\frac{d}{dt}\int_{\mathbb{T}^2}\rho f^{\alpha}\,dx
   +\int_{\mathbb{T}^2}\rho(f+M)^{\frac{\gamma}{\beta}}f^{\alpha-1}\,dx
   \le& C\int_{\mathbb{T}^2}\rho\Big(|\psi|^{\frac{\gamma}{\beta}}+\bigl|\bigl[u_{i},\mathcal{R}_{ij}\bigr](\rho u_{j})\bigr|+|\overline{P}-\overline{G}|\Big)\,f^{\alpha-1}\,dx,\\
   \triangleq& \sum_{i=1}^{3}I_i.
\end{aligned}
\end{equation}
In the following, we estimate the right-hand side term by term.

\noindent\textbf{Step 3: Estimate of $I_{1}$} 

Let us analyze the term
\begin{equation}\begin{aligned}
I_{1} =& \int_{\mathbb{T}^2} \rho \, |\psi|^{\frac{\gamma}{\beta}} \, f^{\alpha-1} \, dx \\
=&\int_{\mathbb{T}^2} \rho \, \Big|\Delta^{-1}\text{div}(\rho(\bm{u}-\overline{\rho \bm{u}}))\Big|^{\frac{\beta}{\gamma}} \, f^{\alpha-1} \, dx+\int_{\mathbb{T}^2} \rho \, \Big|\Delta^{-1}\text{div}(\rho(\overline{\rho \bm{u}}))\Big|^{\frac{\beta}{\gamma}} \, f^{\alpha-1} \, dx\\
\triangleq& I_{11}+I_{12}.
\end{aligned}\end{equation}

We shall estimate \(I_{11}\) by splitting into two terms according to the relation between the pressure exponent \(\gamma\) and the bulk‑viscosity exponent \(2\beta\).

\subsection*{Case 1: \(\gamma \leq 2\beta\)}

Since \(\beta>4/3\) and \(\alpha\ge3\), the quantity \(\frac{\gamma}{\beta}\) is not larger than 2.
We therefore first obtain an \(L^{\infty}\)-bound for \(\psi\) with the aid of elliptic regularity and Sobolev embedding inequality.

Fix a small number \(\varepsilon>0\) satisfying
\begin{equation}
\varepsilon<\begin{cases}
\gamma-1, & \text{if }1<\gamma\le 2,\\[2mm]
\gamma-2, & \text{if }\gamma>2 .\notag
\end{cases}
\end{equation}
The Sobolev embedding inequality give
\begin{equation}
\|\psi\|_{L^\infty}^2\le C(\varepsilon)\,\|\rho (\bm{u}-\overline{\rho \bm{u}})\|_{L^{2+\varepsilon}}^2
\le C(\varepsilon)\,\|\nabla \bm{u}\|_{L^2}^2\,\|\rho\|_{L^{2+2\varepsilon}}^2 .
\end{equation}
If \(\gamma\le2\), choose \(\theta\in(0,1)\) by
\begin{equation}
\frac{1}{2+2\varepsilon}= \frac{1-\theta}{\gamma}+\frac{\theta}{\alpha\beta+1}.\notag
\end{equation}
Because $\varepsilon<\gamma-1$ one has $\frac{2\theta}{\alpha\beta+1}<\frac1\alpha$.
Applying \eqref{rhof} we obtain
\begin{equation}
\|\rho\|_{L^{2+2\varepsilon}}^2\le C\|\rho\|_{L^{\alpha\beta+1}}^{2\theta}
\le C\Bigl(\int_{\mathbb{T}^2}\rho f^{\alpha}\,dx+M^{\alpha}\Bigr)^{\frac1\alpha}.
\end{equation}
If \(\gamma>2\), the choice \(\varepsilon<\gamma-2\) together with the energy estimate \eqref{basicestimate} and \eqref{equzrho} directly yields
\begin{equation}
\|\rho\|_{L^{2+2\varepsilon}}^2\le C\|\rho\|_{L^\gamma}^2\le C .\notag
\end{equation}
Consequently, it follows that
\begin{equation}\label{psi1}
\|\psi\|_{L^\infty}^2\le C\,A_3^2\Bigl(\int_{\mathbb{T}^2}\rho f^{\alpha}\,dx+1\Bigr)^{\frac1\alpha}.
\end{equation}
Insert \eqref{psi1} into $I_{11}$:
\begin{equation}
\begin{aligned}
I_{11}
&\le C\|\psi\|_{L^\infty}^2\int_{\mathbb{T}^2}\rho f^{\alpha-1}\,dx
+C\int_{\mathbb{T}^2}\rho f^{\alpha-1}\,dx \\[1mm]
&\le C\,A_3^2\Bigl(\int_{\mathbb{T}^2}\rho f^{\alpha}\,dx+1\Bigr)
      +C\int_{\mathbb{T}^2}\rho f^{\alpha-1}\,dx .
\end{aligned}
\end{equation}
The last integral can be absorbed by the damping term
\(\displaystyle\int_{\mathbb{T}^2}\rho(f+M)^{\frac{\gamma}{\beta}}f^{\alpha-1}dx\) provided \(M\) is chosen large enough, precisely, taking
\begin{equation}
M>\bigl(C\varepsilon^{-1}\bigr)^{\beta/\gamma},\notag
\end{equation}
we obtain
\begin{equation}\label{3.16}
I_{11}\le C\,A_3^2\Bigl(\int_{\mathbb{T}^2}\rho f^{\alpha}\,dx+1\Bigr)
      +\varepsilon\int_{\mathbb{T}^2}\rho(f+M)^{\frac{\gamma}{\beta}}f^{\alpha-1}\,dx .
\end{equation}

\subsection*{Case 2 \(\gamma>2\beta\)}

When \(\frac{\gamma}{\beta}>2\), we can utilize the integrability of the momentum \(\rho u\) to lower the power of \(\psi\) in exchange for control by the energy norm.

From the basic energy estimate \eqref{basicestimate} we know
\begin{equation}
\sup_{0\le t\le T}\|\rho (\bm{u}-\overline{\rho \bm{u}})\|_{L^{2\gamma/(\gamma+1)}}\le C \sup_{0\le t\le T}\|\rho^{\frac{1}{2}}\|_{L^{2\gamma}}\|\rho^{\frac{1}{2}}(\bm{u}-\overline{\rho \bm{u}})\|_{L^2}\leq C.
\end{equation}
Fix a number $\varepsilon>0$ satisfying
\begin{equation}\nonumber
\frac{\gamma+1}{2\gamma}\Bigl(1-\frac{2\beta}{\gamma}\Bigr)
+\frac{1}{\gamma-\varepsilon}\,\frac{2\beta}{\gamma}<\frac12 .
\end{equation}
Such an $\varepsilon$ exists because for $\gamma>2\beta>2$,
\begin{equation}\nonumber
\frac{\gamma+1}{2\gamma}\Bigl(1-\frac{2\beta}{\gamma}\Bigr)
+\frac{1}{\gamma}\,\frac{2\beta}{\gamma}
= \frac{1}{\gamma}+\frac{1}{2\beta}<1 .
\end{equation}
Applying Sobolev embedding inequality yields
\begin{equation}\label{3.28}
\begin{aligned}
\|\psi\|_{L^\infty}^{\frac{\gamma}{\beta}}
&\le C(\varepsilon)\,
    \|\rho (\bm{u}-\overline{\rho\bm{u}})\|_{L^{2\gamma/(\gamma+1)}}^{\frac{\gamma}{\beta}-2}
    \|\rho(\bm{u}-\overline{\rho\bm{u}})\|_{L^{\gamma-\varepsilon}}^{2} \\
&\le C\|\nabla \mathbf{u}\|_{L^2}^{2}
\le C\,A_3^{2}.
\end{aligned}
\end{equation}
Using \eqref{3.28},
\begin{equation}
\begin{aligned}\label{4.16}
I_{11}
&\le \|\psi\|_{L^\infty}^{\frac{\gamma}{\beta}}
     \int_{\mathbb{T}^2}\rho f^{\alpha-1}\,dx \\[1mm]
&\le C\,A_3^{2}\Bigl(\int_{\mathbb{T}^2}\rho f^{\alpha}\,dx+1\Bigr).
\end{aligned}
\end{equation}
No absorption by the damping term is needed in this case.

Combining \eqref{3.16} and \eqref{4.16} we conclude:
For any \(\varepsilon>0\) there exist constants
$M>\bigl(C\varepsilon^{-1}\bigr)^{\beta/\gamma}$ such that
\begin{equation}\label{3.30}
I_{11}\le C\,A_1^{2}\Bigl(\int_{\mathbb{T}^2}\rho f^{\alpha}\,dx+1\Bigr)
      +\varepsilon\int_{\mathbb{T}^2}\rho(f+M)^{\frac{\gamma}{\beta}}f^{\alpha-1}\,dx .
\end{equation}
For the second part $I_{12}$, we choose $p = \beta\bigl(\alpha-1+\frac{\gamma}{\beta^{2}}\bigr) > 2$. Applying the uniform $L^{p}$ estimate () of the density, we obtain:
\begin{equation}
\begin{aligned}
I_{1,2} &\le \overline{\rho \mathbf{u}}\,\bigl\|\Delta^{-1}\nabla\rho\bigr\|_{L^{\infty}}^{\frac{\gamma}{\beta}}
          \int_{\mathbb{T}^2}\rho\,f^{\alpha-1}\,dx \\
        &\le C\Bigl(\int_{\mathbb{T}^2}\rho^{p\beta+1}\,dx\Bigr)^{\frac{\gamma}{\beta(p\beta+1)}}
          \int_{\mathbb{T}^2}\rho\,f^{\alpha-1}\,dx \\
          &\le C\Bigl(\int_{\mathbb{T}^2}\rho^{p\beta+1}\,dx\Bigr)^{\frac{\gamma}{p\beta^2}}
          \int_{\mathbb{T}^2}\rho\,f^{\alpha-1}\,dx \\
        &\le C\Bigl(\int_{\mathbb{T}^2}\rho\,f^{p}\,dx+M^{p}\Bigr)^{\frac{\gamma}{p\beta^{2}}}
          \int_{\mathbb{T}^2}\rho\,f^{\alpha-1}\,dx .
\end{aligned}
\end{equation}
By choosing the threshold $M$ sufficiently large, specifically
\begin{equation}
\begin{aligned}
M > \Bigl(\frac{C}{\varepsilon}\Bigr)^{\frac{\beta^{2}}{\gamma\beta-\gamma}},\notag
\end{aligned}
\end{equation}
we finally obtain the desired bound:
\begin{equation}
\begin{aligned}\label{3.31}
I_{1,2} \le \varepsilon\int_{\mathbb{T}^2}\rho\,(f+M)^{\frac{\gamma}{\beta}}\,f^{\alpha-1}\,dx .
\end{aligned}
\end{equation}
From a combination of equations \eqref{3.30} and \eqref{3.31}, it follows that
\begin{equation}
\begin{aligned}\label{iall}
I_{1}\le C\,A_3^{2}\Bigl(\int_{\mathbb{T}^2}\rho f^{\alpha}\,dx+1\Bigr)
      +\varepsilon\int_{\mathbb{T}^2}\rho(f+M)^{\frac{\gamma}{\beta}}f^{\alpha-1}\,dx .
\end{aligned}
\end{equation}
\textbf{Step 4: Estimate of $I_{2}$}

We first state a commutator estimate in \eqref{eq:commutator-gradient} that will be used repeatedly.

 To handle $I_2$, we split it into two parts according to the decomposition $\rho \bm{u} = \rho\overline{\rho \bm{u}} + \rho(\bm{u}-\overline{\rho \bm{u}})$:
\begin{equation}
\begin{aligned}
I_2 & = \int_{\mathbb{T}^2}\rho\big|[u_i,\mathcal{R}_{ij}](\rho\,\overline{\rho u})\big|f^{\alpha-1}\,dx \\
    & \quad + \int_{\mathbb{T}^2}\rho\big|[u_i,\mathcal{R}_{ij}]\big(\rho(\bm{u}-\overline{\rho u})\big)\big|f^{\alpha-1}\,dx.
\end{aligned}
\label{eq:J2_split}
\end{equation}

We estimate these two terms separately. For the first term, we take $q = \beta(\alpha-1+1/\beta) > 2$ and apply H\"older's inequality, the commutator estimate \eqref{eq:commutator-gradient}, and the density estimate \eqref{rhof}:

\begin{equation}
\begin{aligned}\label{com1}
& \int_{\mathbb{T}^2}\rho\big|[u_i,\mathcal{R}_{ij}](\rho\,\overline{\rho \bm{u}})\big|f^{\alpha-1}\,dx \\
& \quad\le C\Big(\int_{\mathbb{T}^2}\rho\big|[u_i,\mathcal{R}_{ij}](\rho\,\overline{\rho \bm{u}})\big|^{q}dx\Big)^{1/q}
        \Big(\int_{\mathbb{T}^2}\rho  
        f^{\alpha-1+1/\beta}dx\Big)^{1-1/q} \\
& \quad\le C\Big(\int_{\mathbb{T}^2}\rho^{q+1}dx\Big)^{1/(q^{2}+q)}
        \Big(\int_{\mathbb{T}^2}\big|[u_i,\mathcal{R}_{ij}](\rho\,\overline{\rho \bm{u}})\big|^{q+1}dx\Big)^{1/(q+1)}
        \Big(\int_{\mathbb{T}^2}\rho f^{\alpha-1+1/\beta}dx\Big)^{1-1/q} \\
& \quad\le C\|\nabla \bm{u}\|_{L^{2}}
        \Big(\int_{\mathbb{T}^2}\rho^{q+1}dx\Big)^{1/q}
        \Big(\int_{\mathbb{T}^2}\rho f^{\alpha-1+1/\beta}dx\Big)^{1-1/q}\\
        &\quad\leq C\|\nabla \mathbf{u}\|_{L^{2}}\bigg(\int_{\Omega}\rho f^{\alpha-1+1/\beta}dx+C\bigg)^{1/q}\bigg(\int_{\Omega}\rho f^{\alpha-1+1/\beta}dx\bigg)^{1-1/q}  \\
&\quad\leq C\|\nabla \mathbf{u}\|_{L^{2}}\bigg(\int_{\mathbb{T}^2}\rho f^{\alpha-1+1/\beta}dx+C\bigg)^{1/2}\bigg(\int_{\mathbb{T}^2}\rho f^{\alpha-1+1/\beta}dx\bigg)^{1/2} \\
&\quad\leq C\|\nabla \mathbf{u}\|_{L^{2}}^{2}\bigg(\int_{\mathbb{T}^2}\rho f^{\alpha}dx+1\bigg)+\varepsilon\int_{\mathbb{T}^2}\rho(f+M)^{\frac{\gamma}{\beta}}f^{\alpha-1}dx,
\end{aligned}
\end{equation}
provied $M>(\frac{C}{\epsilon})^{\frac{\beta}{\gamma-1}}.$

For the second term in \eqref{eq:J2_split}, we apply similar techniques. By H\"older's inequality and the commutator estimate again,

\begin{equation}
\begin{aligned}
& \int_{\mathbb{T}^2}\rho\big|[u_i,\mathcal{R}_{ij}]\big(\rho(\bm{u}-\overline{\rho \bm{u}})\big)\big|f^{\alpha-1}\,dx \\
& \quad\le C\Big(\int_{\mathbb{T}^2}\rho\big|[u_i,\mathcal{R}_{ij}]\big(\rho(\bm{u}-\overline{\rho \bm{u}})\big)\big|^{q}dx\Big)^{1/q}
        \Big(\int_{\mathbb{T}^2}\rho f^{\alpha-1+1/\beta}dx\Big)^{1-1/q} \\
& \quad\le C\Big(\int_{\mathbb{T}^2}\rho^{q+1}dx\Big)^{1/(q^{2}+q)}
        \Big(\|\nabla \bm{u}\|_{L^{2}}\|\rho(\bm{u}-\overline{\rho \bm{u}})\|_{L^{q+1}}\Big)
        \Big(\int_{\mathbb{T}^2}\rho f^{\alpha-1+1/\beta}dx\Big)^{1-1/q}.
\end{aligned}
\end{equation}
Using the fact that $ \alpha-1+1/\beta < \alpha$, we obtain

\begin{equation}
\int_{\mathbb{T}^2}\rho\big|[u_i,\mathcal{R}_{ij}]\big(\rho(\bm{u}-\overline{\rho \bm{u}})\big)\big|f^{\alpha-1}\,dx
\le C\|\nabla \bm{u}\|_{L^{2}}^{2}
\Big(\int_{\mathbb{T}^2}\rho f^{\alpha}dx+1\Big).
\label{eq:second_term_final}
\end{equation}
Combining \eqref{com1} and \eqref{eq:second_term_final}, we conclude that,
$$
I_2 \le C A_1^2 \Big(\int_{\mathbb{T}^2}\rho f^{\alpha}dx+1\Big) + \varepsilon\int_{\mathbb{T}^2}\rho(f+M)^{\frac{\gamma}{\beta}}f^{\alpha-1}dx.
$$

\noindent\textbf{Step 5: Estimate of $I_3$}
Applying H\"older inequality  yields
\begin{equation}
\begin{aligned}
|\overline{P}-\overline{G}| &\leq C \left( \int_{\mathbb{T}^2} \rho^{\beta} (\operatorname{div} \bm{u})^2 dx \right)^{\frac{1}{2}} \left( \int_{\mathbb{T}^2} \rho^{\beta} dx \right)^{\frac{1}{2}} + C \\
& \leq C A_3^2\left( \int_{\mathbb{T}^2} \rho f^{\alpha} dx + M^{\alpha} \right)^{\frac{\beta}{\alpha\beta+1}}+C.
\end{aligned}\end{equation}
Now we estimate $I_3$:
\begin{equation}
\begin{aligned}
I_3 &= \int_{\mathbb{T}^2} \rho |\overline{P}-\overline{G}| f^{\alpha-1} dx \\
&\leq C A_3^2 \left( \int_{\mathbb{T}^2} \rho f^{\alpha} dx + 1 \right)^{\frac{\beta}{\alpha\beta+1}} \int_{\mathbb{T}^2} \rho f^{\alpha-1} dx + C \int_{\mathbb{T}^2} \rho f^{\alpha-1} dx\\
&\leq C A_3^2 \left( \int_{\mathbb{T}^2} \rho f^{\alpha} dx + 1 \right)^{\frac{\beta}{\alpha\beta+1} + \frac{\alpha-1}{\alpha}} + C \int_{\mathbb{T}^2} \rho f^{\alpha} dx\\
&\leq C A_3^2 \left( \int_{\mathbb{T}^2} \rho f^{\alpha} dx + 1 \right) + \varepsilon \int_{\mathbb{T}^2} \rho (f + M)^{\frac{\gamma}{\beta}} f^{\alpha-1} dx.
\end{aligned}
\end{equation}
where $M>(C\epsilon)^{\frac{\beta}{\gamma}}.$
This estimate shows that the contribution from $\overline{P}-\overline{G}$ is controlled by the energy dissipation $A_3^2$ and can be absorbed by the damping term, which is crucial for establishing uniform bounds on the density $\rho$.
\begin{equation}
\begin{aligned}
|I_1|+|I_2|+|I_3|\leq C A_3^2 \left( \int_{\mathbb{T}^2} \rho f^\alpha dx + 1 \right) + \varepsilon \int_{\mathbb{T}^2} \rho (f + M)^{\frac{\gamma}{\beta}} f^{\alpha-1} dx.
\end{aligned}
\end{equation}
Substituting this inequality into \eqref{final1} and integrating over $(0,T)$ implies \eqref{eq:lpdensity}.
\end{proof}
Hence we immediately obtain the following lemma.

\begin{lema}\label{cor:rho-Lp-estimate}
For any $c_{*}, c^{*},\alpha \ge 3$, there exist positive constants $C$ independent of $T$ depending only on $\mu,\alpha,\beta,\gamma,\rho_0,\mathbf{u}_0, Z_0$, and $\mathbb{T}^2$ such that
\begin{equation}
\begin{aligned}\label{rhohigh1}
\sup_{0\le t\le T} \|\rho(\cdot,t)\|_{L^{\alpha\beta+1}} \le\; C .
\end{aligned}
\end{equation}
\end{lema}
\begin{proof}
From \eqref{rhof} together with Lemma~\ref{lem:Phi-bound} we obtain
\begin{equation}\label{eq:sup-rho-alpha-beta}
\sup_{0\le t\le T}\int_{\mathbb{T}^2}\rho^{\alpha\beta+1}dx 
\le C\sup_{0\le t\le T}\Bigl(\int_{\mathbb{T}^2}\rho f^{\alpha}dx\Bigr)+C\le C .\notag
\end{equation}
\end{proof}

\begin{lema}\label{basiclogestiamte}
For any $\alpha\in(0,1)$, there is a positive constant $C(\alpha)$ depending only on $c_{*}, c^{*},\alpha,\mu,\beta,\gamma,\|\rho_{0}\|_{L^{\infty}}$, and $\|u_{0}\|_{H^{1}}$ such that
\begin{equation}
\begin{aligned}\label{logestimate}
\sup_{0\leq t\leq T}\log(e+A_{1}^{2}(t)+A_{3}^{2}(t))+\int\limits_{0}^{T}\frac{A_{2}^{2}(t)}{e+A_{1}^{2}(t)}dt\leq C(\alpha)R_{T}^{1+\alpha\beta}.
\end{aligned}
\end{equation}

\end{lema}

\begin{proof}
Direct calculation gives,
\begin{equation}
\nabla^{\perp}\cdot\dot{\mathbf{u}}=\frac{D}{Dt}\omega-(\partial_{1}u\cdot\nabla)u_{2}+(\partial_{2}u\cdot\nabla)u_{1}
=\frac{D}{Dt}\omega+\omega\text{div}\mathbf{u}, 
\end{equation}
and that
\begin{equation}
\begin{aligned}
\text{div}\dot{\mathbf{u}}
&=\frac{D}{Dt}\text{div}\mathbf{u}+(\partial_{1}u\cdot\nabla)u_{1}+(\partial_{2}u\cdot\nabla)u_{2}\\
&=\frac{D}{Dt}\left(\frac{G}{2\mu+\lambda}\right)+\frac{D}{Dt}\left(\frac{P-\overline{P}}{2\mu+\lambda}\right)-2\nabla u_{1}\cdot\nabla^{\perp}u_{2}+(\text{div}u)^{2}.
\end{aligned}
\end{equation}
The momentum equations can therefore be rewritten as
\begin{equation}\label{momequ}
\rho\dot{\mathbf{u}}=\nabla G+\mu\nabla^{\perp}\omega.
\end{equation}
Multiplying \eqref{momequ} by $2\dot{\mathbf{u}}$ and integrating over $\mathbb{T}^2$ leads to
\begin{equation}
\begin{aligned}\label{second}
\frac{d}{dt}A_{1}^{2}+2A_{2}^{2}
&=-\mu\int\omega^{2}\text{div}\mathbf{u}\,dx+4\int G\,\nabla u_{1}\cdot\nabla^{\perp}u_{2}\,dx\\
&\quad-2\int G(\text{div}\mathbf{u})^{2}dx-\int\frac{(\beta-1)\lambda-2\mu}{(2\mu+\lambda)^{2}}G^{2}\text{div}\mathbf{u}\,dx\\
&\quad-2\beta\int\frac{\lambda(P-\overline{P})}{(2\mu+\lambda)^{2}}G\,\text{div}\mathbf{u}\,dx+2\gamma\int\frac{P}{2\mu+\lambda}G\,\text{div}\mathbf{u}\,dx\\
&\quad+2(\gamma-1)\int P\,\text{div}\mathbf{u}\,dx\int\frac{G}{2\mu+\lambda}dx\triangleq\sum_{i=1}^{7}J_{i}.
\end{aligned}
\end{equation}
Define $\varphi_{\alpha}$ 
\begin{equation}\label{varphi}
\varphi_{\alpha}(t)\triangleq 1+A_{1}R_{T}^{\alpha\beta/2}.
\end{equation}
We now estimate each term $J_i$. First, by \eqref{eq:poincare-sobolev},
\begin{equation}\label{w4}
\|\omega\|_{L^{4}}\leq C\|\omega\|_{L^{2}}^{1/2}\|\nabla\omega\|_{L^{2}}^{1/2}\leq C\varphi_{\alpha}^{1/2}\|\omega\|_{H^{1}}^{1/2},
\end{equation}
with $\varphi_{\alpha}$ defined in \eqref{varphi}. Together with \eqref{a3u}  and H\"older's inequality this yields
\begin{equation}
|J_{1}|\leq C\|\omega\|_{L^{4}}^{2}\|\text{div}\mathbf{u}\|_{L^{2}}\leq CA_{3}\|\omega\|_{H^{1}}\varphi_{\alpha}.
\end{equation}
For $I_{2}$ we use an argument from \cite{Desjardins}. Since $\text{rot}\nabla u_{1}=0$ and $\text{div}\nabla^{\perp}u_{2}=0$, \cite{Coifman} gives
\begin{equation}
\|\nabla u_{1}\cdot\nabla^{\perp}u_{2}\|_{{\mathcal H}^{1}}\leq C\|\nabla u\|_{L^{2}}^{2}.\notag
\end{equation}
As $\mathcal{BMO}$ is the dual of ${\mathcal H}^{1}$ (see \cite{C. Fefferman}), we obtain
\begin{equation}
|I_{2}|\leq C\|G\|_{\mathcal{BMO}}\|\nabla u_{1}\cdot\nabla^{\perp}u_{2}\|_{{\mathcal H}^{1}}
\leq C\|\nabla G\|_{L^{2}}\|\nabla \mathbf{u}\|_{L^{2}}^{2}\leq CA_{3}\|G\|_{H^{1}}\varphi_{\alpha},
\end{equation}
where we also used \eqref{a3u} and the bound
\begin{equation}
\begin{aligned}
\|\nabla u\|_{L^{2}}\leq& C\|\omega\|_{L^{2}}+C\|{\rm div}\mathbf{u}\|_{L^{2}}\\
\leq& C\|\omega\|_{L^{2}}+C\left\| \frac{G}{2\mu+\lambda} \right\|_{L^2}+C\left\| \frac{P(Z)-\overline{P(Z)}}{2\mu+\lambda} \right\|_{L^2}\\
\leq& C\varphi_{\alpha}.
\end{aligned}\end{equation}
For $0<\alpha<1$, the terms $J_{3}$--$J_{6}$ can be bounded by H\"older's inequality as
\begin{equation}\begin{aligned}
\sum_{i=3}^{6}|J_{i}|
\leq& C\int\frac{G^{2}|{\rm div}\mathbf{u}|}{2\mu+\lambda}dx+C\int\frac{P+\overline{P}}{2\mu+\lambda}|G||{\rm div}\mathbf{u}|dx\\
\leq& C(\alpha)\left\| \frac{G^2}{2\mu+\lambda} \right\|_{L^2}\varphi_{\alpha}+CA_3\|G\|_{L^{2+\frac{4\gamma}{\beta}}}\left\|\frac{P+1}{2\mu+\lambda} \right\|_{L^{2+\frac{\beta}{\gamma}}},\\
\leq& C(\alpha)A_{3}\|G\|_{H^{1}}\varphi_{\alpha},
\end{aligned}\end{equation}
because, using \eqref{eq:poincare-sobolev} and $\|G\|_{L^{2}}\leq CR_{T}^{\beta/2}A_{1}$,
\begin{equation}\begin{aligned}\label{g4}
\left\|\frac{G^{2}}{\sqrt{2\mu+\lambda}}\right\|_{L^{2}}&\leq \|\frac{G}{\sqrt{2\mu+\lambda}}\|^{1-\alpha}_{L^2}\|G\|^{1+\alpha}_{L^{\frac{2(1+\alpha)}{\alpha}}}\\
&\leq C(\alpha)A_{1}^{1-\alpha}\|G\|_{L^{2}}^{\alpha}\|G\|_{H^{1}}\\
&\leq C(\alpha)A_{1}R_T^{\frac{\alpha\beta}{2}}\|G\|_{H^{1}}\\
&\leq C(\alpha)\|G\|_{H^{1}}\varphi_{\alpha}.
\end{aligned}\end{equation}
Finally, from \eqref{3.12}, \eqref{basicestimate}, and H\"older's inequality,
\begin{equation}
|J_{7}|\leq CA_{3}\|G\|_{L^{2}}+CA_{3}^{2}.
\end{equation}
Inserting these estimates into \eqref{second} gives, for any $\alpha\in(0,1)$,
\begin{equation}\label{step1}
\frac{{\rm d}}{{\rm d}t}A_{1}^{2}(t)+2A_{2}^{2}(t)\leq C(\alpha)A_{3}\left(\|G\|_{H^{1}}+\|\omega\|_{H^{1}}\right)\varphi_{\alpha}+CA_{3}^{2}.
\end{equation}
Notice that from \eqref{momequ},
\begin{equation}\label{gwexpression}
\Delta G={\rm div}(\rho\dot{\mathbf{u}}),\qquad \mu\Delta\omega=\nabla^{\perp}\cdot(\rho\dot{\mathbf{u}}),
\end{equation}
so standard $L^{p}$-elliptic estimates yield for $p\in(1,\infty)$,
\begin{equation}\label{gw}
\|\nabla G\|_{L^{p}}+\|\nabla\omega\|_{L^{p}}\leq C(p,\mu)\|\rho\dot{\mathbf{u}}\|_{L^{p}}.
\end{equation}
Combined with Poincar\'{e}--Sobolev inequality and the bound $|\overline{G}|\leq CA_{3}$, this implies
\begin{equation}\label{wg}
\|\omega\|_{H^{1}}+\|G\|_{H^{1}}\leq CR_{T}^{1/2}A_{2}+CA_{3}.
\end{equation}
Substituting this into \eqref{step1} and applying Cauchy's inequality we finally obtain
\begin{equation}\label{a1b}
\frac{d}{dt} A_1^2(t) + A_2^2(t) \leq 
C(\alpha) R_T \left( 1 + R_T^{\alpha \beta} A_1^2 \right) A_3^2
\leq C(\alpha)R_T^{1+\alpha \beta}A_1^2A_3^2+C(\alpha)R_TA_3^2,
\end{equation}
and dividing the above inequality by $e+A_1^2(t)$ and using \eqref{basicestimate}, we obtain that
\begin{equation}
\sup_{0 \leq t \leq T} \log\left(e + A_1^2(t) \right) + \int_0^T \frac{A_2^2(t)}{e + A_1^2(t)} \, dt \leq C(\alpha) R_T^{1  + \alpha \beta}.
\end{equation}
Combining this with the following inequality, we finally complete the proof of the lemma.
\begin{equation}\label{a1a3}
CA_3^2(t) - C  \leq A_1^2(t) \leq CA_3^2(t) + C.
\end{equation}
\end{proof}
In the subsequent analysis for deriving the upper bound of the density, the $L^\infty$ estimate of the commutator plays an essential role.  This estimate requires a careful and delicate treatment, as it involves the interplay between the singular integral operators, the velocity field, and the density.  The precise control of $\|F\|_{L^\infty}$ in terms of the key quantities $A_1$, $A_2$, $A_3$ and the maximal density $R_T$ is crucial for closing the energy estimates and eventually obtaining a time‑independent bound for the density under the condition $\beta > \frac{3}{2}.$
\begin{lema}\label{commutator estimate}
Let $F$ be defined by
\begin{equation}
\begin{aligned}\label{fdenie}
F\triangleq\sum_{i,j=1}^2\bigl[u_{i},\mathcal{R}_{ij}\bigr](\rho u_{j})
\end{aligned}
\end{equation}
Then for any $\varepsilon > 0$, there exists a constant $C(\varepsilon) > 0$ depending only on $c_{*}, c^{*},\varepsilon$, $\mu$, $\beta$, $\gamma$, $\|\rho_0\|_{L^\infty}$, and $\|u_0\|_{H^1}$ such that
\begin{equation}
\begin{aligned}
\|F\|_{L^\infty} \leq \frac{C(\varepsilon) R_T^{-1} A_2^2}{e+A_1^2}
+ C(\varepsilon) A_3^2 R_T^{3/2+\varepsilon}
+ C(\varepsilon) R_T^{1+\varepsilon}.
\end{aligned}
\end{equation}
\end{lema}
\begin{proof}
We start from the $L^q$-estimate of $F$. By Lemma \ref{lem:commutator} and \eqref{a3u}, we have 
\begin{equation}
\begin{aligned}\nonumber
\|F\|_{L^q} \le C(q) \|\nabla \mathbf{u}\|_{L^2} \|\rho \mathbf{u}\|_{L^q}
\le C(q) R_T (A_3^2+1).
\end{aligned}
\end{equation}
Using the Gagliardo--Nirenberg inequality and Lemma \ref{lem:commutator} again, we obtain for $q\in(8,\infty)$,
\begin{equation}
\begin{aligned}\label{ffina}
\|F\|_{L^\infty}
&\le C(q)\|F\|_{L^q} + C(q)\|F\|_{L^q}^{1-4/q}\|\nabla F\|_{L^{4q/(q+4)}}^{4/q} \\
&\le C(q)R_T(A_3^2+1) + C(q)\|\nabla \mathbf{u}\|_{L^2}^{1-4/q}\|\rho \mathbf{u}\|_{L^q}^{1-4/q}\|\nabla \mathbf{u}\|_{L^4}^{4/q}\|\rho \mathbf{u}\|_{L^q}^{4/q} \\
&\le C(q)R_T(A_3^2+1) + C(q)A_3^{1-4/q}\|\nabla \mathbf{u}\|_{L^4}^{4/q}\|\rho \mathbf{u}\|_{L^q}.
\end{aligned}
\end{equation}
Next we estimate $\|\nabla \mathbf{u}\|_{L^4}$ and $\|\rho \mathbf{u}\|_{L^q}$. 

From \eqref{w4}, \eqref{g4}, \eqref{a1a3} we have
\begin{equation}
\begin{aligned}\label{u4}
\|\nabla \mathbf{u}\|_{L^{4}} &\leq C\left(\|{\rm div}\mathbf{u}\|_{L^{4}}+\|{\omega}\|_{L^{4}}\right) \\
&\leq C\left\| \frac{G}{2\mu+\lambda} \right\|_{L^{4}} 
+ C\left\| \frac{P-\bar{P}}{2\mu+\lambda} \right\|_{L^{4}} 
+ C\|{\omega}\|_{L^{4}} \\
&\leq C\left\| \frac{G^{2}}{2\mu+\lambda} \right\|_{L^{2}}^{1/2} 
+ C\varphi_{\alpha}^{1/2}\|{\omega}\|_{H^{1}}^{1/2} +C\\
&\leq C(\alpha)\left(\|G\|_{H^{1}}+\|{\omega}\|_{H^{1}}\right)^{1/2}\varphi_{\alpha}^{1/2}+C \\
&\leq C R_{T}^{\hat{C}}(e+A_{3})\left( \frac{R_{T}^{-4}A_{2}^{2}}{e+A_{1}^{2}} \right)^{1/4} 
+ C R_{T}^{\hat{C}}(e+A_{3}).
\end{aligned}
\end{equation}
for some $\hat{C}>1$ depending only on $\beta,\gamma$.

\noindent Using \eqref{eq:brezis-wainger} and \eqref{a3u}, we get for $\alpha\in(0,1)$,
\begin{equation}
\begin{aligned}\label{infty}
\|\mathbf{u}\|_{L^\infty}
\le& C\|\nabla \mathbf{u}\|_{L^2}\log^{1/2}\bigl(e+\|\nabla \mathbf{u}\|_{L^4}\bigr)
+ C\|\nabla \mathbf{u}\|_{L^2}+C\\
\le& C(\alpha)A_3R_T^{(1+\alpha\beta)/2}
+ C\Bigl( \frac{R_T^{-4}A_2^2}{(e+A_1^2)(e+A_3)^4} \Bigr)^{1/2}+C.
\end{aligned}
\end{equation}
Hence, for $q>8$ and $\alpha\in(0,1)$,
\begin{equation}
\begin{aligned}\label{rhouq}
\|\rho \mathbf{u}\|_{L^q}
&\le C R_T^{1-1/q}\|\rho^{1/2}\mathbf{u}\|_{L^2}^{2/q}\|\mathbf{u}\|_{L^\infty}^{1-2/q} \\
&\le C(\alpha) A_3^{1-2/q}R_T^{(3+\alpha\beta)/2}
+ C\Bigl( \frac{R_T^{-1}A_2^2}{(e+A_1^2)(e+A_3)^4} \Bigr)^{1/2-1/q}
+ C R_T.
\end{aligned}
\end{equation}
We proceed to estimate the following term. To this end, we invoke \eqref{rhouq}
\begin{equation}
\begin{aligned}
A_3^{1-4/q}\|\nabla \mathbf{u}\|_{L^4}^{4/q}\|\rho \mathbf{u}\|_{L^q}\leq& C(q,\alpha)\|\nabla \mathbf{u}\|_{L^{4}}^{4/q}A_{3}^{(2q-6)/q}R_{T}^{(3+\alpha\beta)/2} \\
&\quad +C(q,\alpha)A_{3}^{(q-4)/q}\|\nabla \mathbf{u}\|_{L^{4}}^{4/q}\left(\frac{R_{T}^{-1}A_{3}^{2}}{(e+A_{1}^{2})(e+A_{3})^{4}}\right)^{1/2-1/q} \\
&\quad +C(q)R_{T}A_{3}^{1-4/q}\|\nabla \mathbf{u}\|_{L^{4}}^{4/q} \\
 \triangleq& \sum_{i=1}^{3}M_{i}.
 \end{aligned}
\end{equation}
An application of the H\"{o}lder inequality gives
\begin{equation}
\begin{aligned}\label{m1}
|M_1| &\leq C(q,\alpha)\left(R_{T}^{3+\alpha\beta)/2+4\tilde{C}/q}A_{3}^{2-6/q}\right)^{q/(q-3)} \\
&\quad +C(q,\alpha)(e+A_{3}^{2})+\frac{R_{T}^{-1}A_{3}^{2}}{e+A_{1}^{2}} \\
&\leq C(q,\alpha)+C(q,\alpha)A_{3}^{2}R_{T}^{\tilde{\kappa}(\alpha,q)}+C(q,\alpha)\frac{R_{T}^{-1}A_{3}^{2}}{e+A_{1}^{2}},
\end{aligned}
\end{equation}
with
$$\tilde{\kappa}(\alpha,q)\triangleq\Big(\frac{3+\alpha\beta}{2}+\frac{4\tilde{C}}{q}\Big)\frac{1}{1-\frac{3}{q}}$$
By applying \eqref{u4}, we have
\begin{equation}
\begin{aligned}\label{m2m3}
|M_2|+|M_3|\leq& C(q,\alpha)\|\nabla \mathbf{u}\|^{\frac{4}{q}}_{L^4}\left(\frac{R_{T}^{-1}A_{3}^{2}}{e+A_{1}^{2}}\right)^{1/2-1/q}+C(q,\alpha)R_{T}(A_{3}+1)\|\nabla \mathbf{u}\|_{L^{4}}^{4/q}\\
\leq& C(q,\alpha)R_{T}\|\nabla \mathbf{u}\|_{L^{4}}^{8/q}+C(q,\alpha)\frac{R_{T}^{-1-\kappa}A_{3}^{2}}{e+A_{1}^{2}}+C(q,\alpha)R_{T}+C(q,\alpha)R_{T}A_{3}^{2} \\
\leq& C(q,\alpha)R_{T}^{1+16\tilde{C}/q}+C(q,\alpha)R_{T}^{1+16\tilde{C}/q}A_{3}^{2}+C(q,\alpha)\frac{R_{T}^{-1-\kappa}A_{3}^{2}}{e+A_{1}^{2}},
\end{aligned}
\end{equation}
Therefore, combining equations \eqref{ffina}, \eqref{m1}, and \eqref{m2m3}, and choosing $q$ sufficiently large and $\alpha$ sufficiently small, we complete the proof of the Lemma \ref{commutator estimate}.
\end{proof}
\begin{lema}\label{rhodependent}
Assume that \eqref{beta23}, there exists a positive constant $C$ depending only on $c_{*}, c^{*},\mu, \beta, \gamma, \|\rho_0\|_{L^\infty}$, and $\|\mathbf{u}_0\|_{H^1}$ such that
\begin{equation}
\begin{aligned}\label{independestimate}
\sup_{0 \leq t \leq T}\left(\|\rho\|_{L^\infty} + \|\mathbf{u}\|_{H^1}\right) 
+ \int_0^T \left( \|\omega\|_{H^1}^2 + \|G\|_{H^1}^2 + A_2^2(t) \right) dt \leq C.
\end{aligned}
\end{equation}
\end{lema}
\begin{proof}
From equation \eqref{changeequ}, we can define
\begin{equation}
\begin{aligned}
y = \theta(\rho), \qquad g(y) = -(\theta^{-1}(y))^{\gamma}\theta_0^{\gamma}, \qquad 
h(t) = -\psi(t) + \int_0^t (\overline{P} -\overline{G} -F) \, ds.
\end{aligned}
\end{equation}
Then \eqref{changeequ} can be rewritten as the ordinary differential equation along the particle trajectory:
\begin{equation}
\begin{aligned}
\frac{D}{Dt} y = g(y) + \frac{D}{Dt} h.
\end{aligned}
\end{equation}
First, from \eqref{a3u}, \eqref{logestimate} and Lemma \ref{lem:poincare-sobolev} we have
\begin{equation}
\begin{aligned}\label{psi}
\|\psi\|_{L^{\infty}} 
    \leq& C\|\nabla\psi\|_{L^{2}}\log^{1/2}(e+\|\nabla\psi\|_{L^{3}})+C\|\psi\|_{L^{2}}+C \\
    \leq& C\|\rho \mathbf{u}\|_{L^{2}}\log^{1/2}(e+\|\rho \mathbf{u}\|_{L^{3}})+C\|\rho \mathbf{u}\|_{L^{2\gamma/(\gamma+2)}}+C \\
    \leq& CR_{T}^{1/2}\log^{1/2}(e+R_{T}(1+\|\nabla \mathbf{u}\|_{L^{2}}))+C \\
    \leq& CR_{T}^{1/2}\log^{1/2}(e+A_{3}^{2})+CR_{T}\\
    \leq&CR_{T}^{\frac{2+\alpha\beta}{2}}
\end{aligned}
\end{equation}
Second, using \eqref{basicestimate} gives
\begin{equation}
\begin{aligned}\label{pg}
|\overline{P} - \overline{G}| \leq C + C A_3^2.
\end{aligned}
\end{equation}
Third, Lemma \ref{commutator estimate} together with \eqref{logestimate} and \eqref{basicestimate}   implies that for any \(\varepsilon>0\),
\begin{equation}
\begin{aligned}\nonumber
\int_{t_1}^{t_2} \|F\|_{L^\infty} \, ds 
\leq C(\varepsilon) R_T^{3/2 + \varepsilon} 
+ C(\varepsilon) R_T^{1+\varepsilon} (t_2 - t_1).
\end{aligned}
\end{equation}
Combining \eqref{pg}, \eqref{psi} and the above estimate, we obtain for any $0 \leq t_1 < t_2 \leq T$
\begin{equation}
\begin{aligned}\nonumber
|h(t_2) - h(t_1)| 
&\leq \|\psi(t_2) - \psi(t_1)\|_{L^\infty} 
+ \int_{t_1}^{t_2} |\overline{P} - \overline{G}| \, ds 
+ \int_{t_1}^{t_2} \|F\|_{L^\infty} \, ds \\
&\leq C(\varepsilon) R_T^{3/2}
+ C(\varepsilon) R_T^{1+\varepsilon} (t_2 - t_1).
\end{aligned}
\end{equation}
Hence, in the notation of Lemma \ref{lem:Zlotnik} we may take
\begin{equation}
\begin{aligned}\nonumber
N_0 = C(\varepsilon) R_T^{\frac{3}{2}}, \qquad
N_1 = C(\varepsilon) R_T^{1+\varepsilon}.
\end{aligned}
\end{equation}
For the function , we have
\begin{equation}
\begin{aligned}\nonumber
g(\zeta) \leq - N_1 \quad \text{for all} \quad 
\zeta \geq \bar{\zeta} := C(\varepsilon,c_{*}) R_T^{\beta(1+\varepsilon)/\gamma}.
\end{aligned}
\end{equation}
Applying Lemma \ref{lem:Zlotnik} we deduce
\begin{equation}
\begin{aligned}\label{3.55}
R_T^\beta\leq C(\varepsilon,c_{*}) R_T^{\max\{3/2, \;  (1+\varepsilon)\beta/\gamma\}}.
\end{aligned}
\end{equation}
 Under the assumption \eqref{beta23}, a direct computation shows that for sufficiently small $\varepsilon > 0$ the exponent on the right-hand side of \eqref{3.55} is strictly smaller than $\beta$. Consequently, we obtain the uniform bound
\begin{equation}
\begin{aligned}\label{3.56}
\sup_{0 \leq t \leq T} \|\rho(\cdot, t)\|_{L^\infty} \leq C.
\end{aligned}
\end{equation}
With \eqref{3.56} in hand, we can improve the previous estimates. Combining equations \eqref{3.56}, \eqref{a1b}, we have
\begin{equation}
\begin{aligned}\label{orderestimate}
\sup_{0 \leq t \leq T} A_1^2(t) + \int_0^T A_2^2(t) \, dt \leq C.
\end{aligned}
\end{equation}
From \eqref{wg} and \eqref{orderestimate} we also obtain
\begin{equation}
\begin{aligned}
\int_0^T \left( \|\omega\|_{H^1}^2 + \|G\|_{H^1}^2 \right) dt \leq C.
\end{aligned}
\end{equation}
Finally, the bound on $\|\mathbf{u}\|_{H^1}$ follows from \eqref{a1b} and \eqref{basicestimate}.  Collecting \eqref{wg}, \eqref{independestimate}  we arrive at \eqref{independestimate}, which completes the proof of Lemma \ref{rhodependent}.
\end{proof}
\subsubsection{Time-independent upper bound of the density}
\begin{lema}
 Let $\nu\triangleq R_T^{-\frac{\beta}{2}}\nu_0$ with $\nu_0\in(0,1)$ sufficiently small depending only on $\mu$ and $\mathbb{T}^2$. Then there exists a constant $C$ depending on $\mathbb{T}^2$, $T,\,c_{*}, c^{*},\mu$, $\beta$, $\gamma$, $\|\rho_0\|_{L^\infty}$, $\|Z_0\|_{L^\infty}$ and $\|u_0\|_{H^1}$ such that
\begin{equation}
\begin{aligned}\label{rhou2}
\sup_{0\leq t\leq T}\int_{\mathbb{T}^2}\rho|\mathbf{u}|^{2+\nu}dx\leq C(T).
\end{aligned}
\end{equation}
\end{lema}
\begin{proof}
Multiplying the equation $\eqref{trans1}_2$ by $(2+\nu)|\mathbf{u}|^{\nu}\mathbf{u}$, and integrating over $\mathbb{T}^2$, we obtain
\begin{align*}
&\frac{\mathrm{d}}{\mathrm{d}t}\int_{\mathbb{T}^2}\rho|\mathbf{u}|^{2+\nu}dx +(2+\nu)\int_{\mathbb{T}^2}|\mathbf{u}|^{\nu}\left(\mu|\nabla \mathbf{u}|^{2}+(\mu+\rho^{\beta})(\mathrm{div}\mathbf{u})^{2}+\mu\nu\nabla|\mathbf{u}||^2\right)dx\\
&\leq(2+\nu)\nu\int(\mu+\rho^{\beta})|\mathrm{div}\mathbf{u}||\mathbf{u}|^{\nu}|\nabla \mathbf{u}|dx+C\int_{\mathbb{T}^2}\rho^{\gamma}|\mathbf{u}|^{\nu}|\nabla \mathbf{u}|dx\\
&\leq\frac{2+\nu}{2}\int(\mu+\rho^{\beta})(\mathrm{div}\mathbf{u})^{2}|\mathbf{u}|^{\nu}dx+\frac{2+\nu}{2}\nu_{0}^{2}(\mu+1)\int|\mathbf{u}|^{\nu}|\nabla u|^{2}dx\\
&\quad+\mu\int_{\mathbb{T}^2}|\mathbf{u}|^{\nu}|\nabla \mathbf{u}|^{2}dx+C\int_{\mathbb{T}^2}\rho|\mathbf{u}|^{2+\nu}dx+C\int_{\mathbb{T}^2}\rho^{(2+\nu)\gamma-\nu/2}dx,
\end{align*}
provided $\nu_0(\mu)$ is taken sufficiently small, Gronwall's inequality together with \eqref{rhohigh1} gives \eqref{rhou2}.
\end{proof}

\begin{lema}
\label{lem:commutator_Linf}
Assume $\beta > 1$. For any $\varepsilon > 0$, there exists a positive constant $C(\varepsilon, T)$ depending only on $\varepsilon$, $T$, $\mu$, $\beta$, $\gamma$, $\|\rho_0\|_{L^\infty}$, $\|Z_0\|_{L^\infty}$ and $\|\mathbf{u}_0\|_{H^1}$ such that
\begin{equation}\label{commutatornew}
\|F\|_{L^\infty} \le \frac{C(\varepsilon, T) A_2^2}{e + A_1^2} + C(\varepsilon, T)(1 + A_3^2) R_T^{1 + \frac{\beta}{4} + \varepsilon}.
\end{equation}
\end{lema}

\begin{proof}
Given that $q > 4$, define $r\triangleq (q - 2)(2 + \nu)/\nu>2$.   
Applying H\"older's inequality together with estimates \eqref{rhou2}, \eqref{basicestimate} and \eqref{eq:poincare-type}, we obtain
\begin{equation}
\begin{aligned}\label{rhouqnew}
\|\rho \mathbf{u}\|_{L^q} &\leq C\|\rho(\mathbf{u}-\overline{\rho \mathbf{u}})\|_{L^q}+\|\rho\|_{L^q}\\
    &\leq C \|\rho (\mathbf{u}-\overline{\rho \mathbf{u}})\|_{L^{2+\nu}}^{2/q} \|\rho (\mathbf{u}-\overline{\rho \mathbf{u}})\|_{L^r}^{1-2/q} +\|\rho\|_{L^q}\\
    &\leq CR_T \|(\mathbf{u}-\overline{\rho \mathbf{u}})\|_{L^r}^{1-2/q} +\|\rho\|_{L^q}\\
    &\leq CR_T \bigl( r^{1/2} \|\mathbf{u}\|_{H^1} \bigr)^{1-2/q} \\
    &\leq C(q) \, R_T^{1+\frac{\beta(q-2)}{4q}} \, \bigl( 1+\|\nabla \mathbf{u}\|_{L^2} \bigr)^{1-2/q}.
\end{aligned}
\end{equation}
We start from the pointwise estimate for the commutator derived in \eqref{ffina}, which holds for any $q > 8$:
\begin{equation}\label{ffanew}
\|F\|_{L^\infty} \le C(q) R_T (A_3^2 + 1) + C(q) A_3^{1 - \frac{4}{q}} \|\nabla \mathbf{u}\|_{L^4}^{\frac{4}{q}} \|\rho \mathbf{u}\|_{L^q}.
\end{equation}
Based on \eqref{rhouqnew}, we are required to estimate the commutator term. The key step in this estimation involves controlling the following quantity:
\begin{equation}
\begin{aligned}
& A_3^{\frac{q-4}{q}} \|\nabla \mathbf{u}\|_{L^4}^{\frac{4}{q}} \|\rho \mathbf{u}\|_{L^q} \\
&\quad \le C(q, T) R_T^{1 + \frac{\beta(q-2)}{4q}} \bigl( A_3^{2 - \frac{6}{q}} + 1 \bigr) \|\nabla \mathbf{u}\|_{L^4}^{\frac{4}{q}} \\
&\quad \le C(q, T) R_T^{1 + \frac{\beta(q-2)}{4q} + \frac{4\tilde{C}}{q}} \bigl( A_3^{2 - \frac{2}{q}} + 1 \bigr) \Bigl( \frac{A_2^2}{e + A_1^2} \Bigr)^{\frac{1}{q}} \\
&\qquad + C(q, T) R_T^{1 + \frac{\beta(q-2)}{4q} + \frac{4\tilde{C}}{q}} \bigl( A_3^2 + 1 \bigr) \\
&\quad \le C(q, T) R_T^{1 + \frac{\beta}{4} + \frac{8\tilde{C}}{q-1}} (1 + A_3^2) + \frac{A_2^2}{e + A_1^2},
\end{aligned}
\end{equation}
where $\tilde{C} > 1$ is a constant depending only on $\beta$ and $\gamma$.

Inserting this estimate into \eqref{ffanew} gives
\begin{equation}
\|F\|_{L^\infty} \le C(q, T) R_T^{1 + \frac{\beta}{4} + \frac{8\tilde{C}}{q-1}} (1 + A_3^2) + C(q, T) \frac{A_2^2}{e + A_1^2}.
\end{equation}

Finally, by choosing $q$ sufficiently large so that $\frac{8\tilde{C}}{q-1} < \varepsilon$, we obtain the desired bound \eqref{commutatornew}.
This completes the proof of Lemma \ref{lem:commutator_Linf}.
\end{proof}

\begin{prop}\label{propdependentt}
Assume \eqref{beta23} holds. Then there exists a constant $C(T) > 0$ such that
\begin{equation}\label{estimateoft}
\sup_{0 \le t \le T} \big( \|\rho(t)\|_{L^{\infty}} + \|\mathbf{u}(t)\|_{H^1} \big) 
+ \int_0^T \big( \|\omega(t)\|_{H^1}^2 + \|G(t)\|_{H^1}^2 + A_2^2(t) \big) \, dt \le C(T),
\end{equation}.
\end{prop}
\begin{proof}
From Lemmas \ref{basiclogestiamte} and Lemma \ref{lem:commutator_Linf}, we obtain
\begin{equation}
R_T^\beta\le C(\varepsilon,T)R_T^{\max\{1+\beta/4+\varepsilon,\;4/3\}}.
\end{equation}
Since $\beta>4/3$, choosing $\varepsilon$ sufficiently small yields
\begin{equation}
\sup_{0\le t\le T}\|\rho(t)\|_{L^\infty}\le C(T).
\end{equation}
Finally, combining \eqref{a1b}, \eqref{basicestimate}, \eqref{wg}, and Gronwall's inequality, we obtain \eqref{propdependentt} completes the proof of Prop \ref{propdependentt}.
\end{proof}
\subsection{A priori estimates (II): higher order estimates} 
\begin{lema}\label{timehighorderestimate}
Assume that there exists a positive constant $M$ such that
\begin{equation}\label{rhobound}
\sup_{0 \leq t \leq T} (\| \rho(\cdot, t) \|_{L^\infty}+\| Z(\cdot, t) \|_{L^\infty}) \leq M.
\end{equation}
Then, there exists a positive constant $C(M)$, depending only on $c_{*}, c^{*}, M, \mu, \beta, \gamma,$ and $\| \mathbf{u}_0 \|_{H^1}$, such that
\begin{equation}\label{highestimate}
\sup_{0 \leq t \leq T} \sigma(t) \int_{\mathbb{T}^2} \rho |\dot{\mathbf{u}}|^2 \, dx + \int_0^T \sigma(t) \| \nabla \dot{\mathbf{u}} \|_{L^2}^2 \, dt \leq C(M),
\end{equation}
where  $\sigma(t) = \min\{1, t\}$. Furthermore, for any $p \in [1, \infty)$, there exists a positive constant  $C(p, M)$, depending only on $c_{*}, c^{*},p, M, \mu, \beta, \gamma, \| \rho_0 \|_{L^\infty}$, and  $\| \mathbf{u}_0 \|_{H^1}$, such that
\begin{equation}\label{nablau}
\sup_{0 \leq t \leq T} \sigma(t)^{1/2} \| \nabla \mathbf{u} \|_{L^p} \leq C(p, M).
\end{equation}
\end{lema}

\begin{proof}
Initially, combining \eqref{a1b}, \eqref{basicestimate}, \eqref{a3u} and \eqref{rhobound} yields
\begin{equation}\label{uh1}
\sup_{0 \leq t \leq T} \|\mathbf{u}\|_{H^{1}} + \int\limits_{0}^{T} \Big( \|\nabla \mathbf{u}\|_{L^{2}}^{2} + \|\rho^{1/2} \dot{\mathbf{u}}\|_{L^{2}}^{2} \Big) dt \leq C(M). 
\end{equation}

To derive \eqref{rhobound}, we adopt an approach originating from Hoff \cite{hoff}. Acting with the operator $\dot{u}^{j} [\partial/\partial t + \operatorname{div}(\mathbf{u} \cdot)]$ on the $j$-th component of $\eqref{trans1}_2$, summing over $j$, and integrating over $\mathbb{T}^{2}$ results after integration by parts in
\begin{equation}
\begin{aligned}\label{nfin}
\Big( \frac{1}{2} \int \rho |\dot{u}|^{2} dx \Big)_{t}
&= - \int \dot{u}_{j} \big[ \partial_{j} P_{t} + \operatorname{div}(\partial_{j} P \mathbf{u}) \big] dx \\
&\quad + \mu \int \dot{u}_{j} \big[ \partial_{t} \Delta u_{j} + \operatorname{div}(\mathbf{u} \Delta u_{j}) \big] dx \\
&\quad + \int \dot{u}_{j} \big[ \partial_{jt} ((\mu + \lambda) \operatorname{div} \mathbf{u}) + \operatorname{div}(u \partial_{j} ((\mu + \lambda) \operatorname{div} \mathbf{u})) \big] dx \\
&\triangleq N_{1} + N_{2} + N_{3}. 
\end{aligned}
\end{equation}
For the first term, employing the continuity equation \eqref{trans1}$_1$ and integration by parts gives
\begin{equation}
\begin{aligned}\label{n1}
N_{1} 
&= - \int \dot{u}_{j} \big[ \partial_{j} P_{t} + \operatorname{div}(\partial_{j} P \mathbf{u}) \big] dx \\
&= \int \big[ -P'(Z) Z \operatorname{div} \mathbf{u} \, \partial_{j} \dot{u}_{j} + \partial_{k} (\partial_{j} \dot{u}_{j} u_{k}) P - P \partial_{j} (\partial_{k} \dot{u}_{j} u_{k}) \big] dx \\
&\leq C(M) \|\nabla \mathbf{u}\|_{L^{2}} \|\nabla \dot{\mathbf{u}}\|_{L^{2}} \\
&\leq \frac{\mu}{8} \|\nabla \dot{\mathbf{u}}\|_{L^{2}}^{2} + C(M) \|\nabla \mathbf{u}\|_{L^{2}}^{2}. 
\end{aligned}
\end{equation}

The second term, after similar integration by parts, becomes
\begin{equation}
\begin{aligned}\label{n2}
N_{2} 
&= \mu \int \dot{u}_{j} \big[ \partial_{t} \Delta u_{j} + \operatorname{div}(u \Delta u_{j}) \big] dx \\
&= - \mu \int \big( |\nabla \dot{\mathbf{u}}|^{2} + \partial_{i} \dot{u}_{j} \partial_{k} u_{k} \partial_{i} u_{j} - \partial_{i} \dot{u}_{j} \partial_{i} u_{k} \partial_{k} u_{j} - \partial_{i} \dot{u}_{j} \partial_{i} u_{k} \partial_{k} u_{j} \big) dx \\
&\leq - \frac{3\mu}{4} \int |\nabla \dot{\mathbf{u}}|^{2} dx + C(M) \int |\nabla \mathbf{u}|^{4} dx. `
\end{aligned}
\end{equation}

Concerning the third term, analogous manipulations lead to
\begin{eqnarray}
\label{n3}
N_{3} 
&=& \int \dot{u}_{j} \big[ \partial_{jt} ((\mu + \lambda) \operatorname{div} \mathbf{u}) + \operatorname{div}(u \partial_{j} ((\mu + \lambda) \operatorname{div} \mathbf{u})) \big] dx \nonumber\\
&=& - \int \partial_{j} \dot{u}_{j} \big[ ((\mu + \lambda) \operatorname{div} \mathbf{u})_{t} + \operatorname{div}(\mathbf{u} (\mu + \lambda) \operatorname{div} \mathbf{u}) \big] dx \nonumber \\
&&\quad - \int \dot{u}_{j} \operatorname{div} \big( \partial_{j} u (\mu + \lambda) \operatorname{div} \mathbf{u} \big) dx \nonumber\\
&=& - \int \Big( \frac{D}{Dt} \operatorname{div} \mathbf{u} + \partial_{j} u_{i} \partial_{i} u_{j} \Big) \Big[ (\mu + \lambda) \frac{D}{Dt} \operatorname{div} \mathbf{u} - \rho \lambda'(\rho) \operatorname{div} \mathbf{u} \Big] dx \\
&&\quad + \int \nabla \dot{u}_{j} \cdot \partial_{j} u (\mu + \lambda) \operatorname{div} \mathbf{u} dx \nonumber\\
&&\leq - \frac{\mu}{2} \int \Big( \frac{D}{Dt} \operatorname{div} \mathbf{u} \Big)^{2} dx + \frac{\mu}{8} \|\nabla \dot{\mathbf{u}}\|_{L^{2}}^{2} + C(M) \|\nabla \mathbf{u}\|_{L^{4}}^{4} + C(M) \|\nabla \mathbf{u}\|_{L^{2}}^{2}, \nonumber
\end{eqnarray}
where the third equality uses the identity
\begin{equation}
\begin{aligned}
\big( (\mu + \lambda) \operatorname{div} \mathbf{u} \big)_{t} + \operatorname{div}(\mathbf{u} (\mu + \lambda) \operatorname{div} \mathbf{u})
= (\mu + \lambda) \frac{D}{Dt} \operatorname{div} \mathbf{u} - \rho \lambda'(\rho) \operatorname{div} \mathbf{u},
\end{aligned}
\end{equation}
a direct consequence of \eqref{trans1}$_1$.

Inserting the estimates \eqref{n1}-\eqref{n3} into \eqref{nfin} produces
\begin{equation}
\begin{aligned}\label{na}
\Big( \int \rho |\dot{\mathbf{u}}|^{2} dx \Big)_{t} 
&+ \mu \int |\nabla \dot{\mathbf{u}}|^{2} dx + \mu \int \Big( \frac{D}{Dt} \operatorname{div} \mathbf{u} \Big)^{2} dx \\
&\leq C(M) \|\nabla \mathbf{u}\|_{L^{4}}^{4} + C(M) \|\nabla \mathbf{u}\|_{L^{2}}^{2} \\
&\leq C(M) \big( \|G\|_{L^{4}}^{4} + \|\omega\|_{L^{4}}^{4} + \|P - \overline{P}\|_{L^{4}}^{4} + \|\nabla \mathbf{u}\|_{L^{2}}^{2} \big) \\
&\leq C(M) \big( \|G\|_{L^{2}}^{2} \|G\|_{H^{1}}^{2} + \|\omega\|_{L^{2}}^{2} \|\nabla \omega\|_{L^{2}}^{2} + \|P - \overline{P}\|_{L^{2}}^{2} + \|\nabla \mathbf{u}\|_{L^{2}}^{2} \big) \\
&\leq C(M) \|\rho^{1/2} \dot{\mathbf{u}}\|_{L^{2}}^{2} + C(M) \|\nabla \mathbf{u}\|_{L^{2}}^{2}, 
\end{aligned}
\end{equation}
the last inequality relying on \eqref{wg}, \eqref{a3u} together with \eqref{rhobound}. 

Multiplying \eqref{na} by $\sigma(t) = \min\{1, t\}$ and integrating over $(0, T)$ finally delivers \eqref{highestimate}  thanks to \eqref{uh1}.

To establish \eqref{nablau}, we observe that for any $p \ge 2$,
\begin{equation}
\begin{aligned}
\|\nabla \mathbf{u}\|_{L^{p}} 
&\leq C(p) \|\operatorname{div} \mathbf{u}\|_{L^{p}} + C(p) \|\omega\|_{L^{p}} \\
&\leq C(p) \|G\|_{L^{p}} + C(p) \|P - \overline{P}\|_{L^{p}} + C(p) \|\omega\|_{L^{p}} \\
&\leq C(p) \|G\|_{H^{1}} + C(p) \|\omega\|_{H^{1}} + C(p, M) \\
&\leq C(p, M) \|\rho^{1/2} \dot{\mathbf{u}}\|_{L^{2}} + C(p, M),
\end{aligned}
\end{equation}
where the final bound follows from \eqref{wg} and \eqref{a3u}. Combining this with \eqref{highestimate} yields \eqref{nablau}, thereby concluding the proof of Lemma \ref{timehighorderestimate}.
\end{proof}

\begin{lema}\label{rhohigh}Suppose that \eqref{beta43} is valid. There exists a positive constant $C$, depending only on $T, c_{*}, c^{*},q, \mu, \gamma, \beta, \|\mathbf{u}_0\|_{H^1}$, and $\|\rho_0\|_{W^{1,q}},\,\|Z_0\|_{W^{1,q}}$, such that
\begin{equation}\begin{aligned}\label{rhozhihg}
\sup_{0\leq t\leq T}\Big(\|\rho\|_{W^{1,q}}+\|Z\|_{W^{1,q}}+\| \mathbf{u}\|_{H^1}\Big)\leq C, 
\end{aligned}\end{equation}
for any $q>2$.
\end{lema}
\begin{proof}
For $r\in[2,q]$, note that~$|\nabla\rho|^r$~satisfies
\begin{equation}\begin{aligned}\label{10.8}
&(|\nabla\rho|^r)_t+\text{div}(|\nabla\rho|^r \mathbf{u})+(r-1)|\nabla\rho|^r\text{div}\mathbf{u}+r|\nabla\rho|^{r-2}(\nabla\rho)^{T}\nabla \mathbf{u}(\nabla\rho)\\
&+r\rho|\nabla\rho|^{r-2}\nabla\rho\cdot{\nabla\text{div}\mathbf{u}}=0.
\end{aligned}\end{equation}
Integrating the resulting equation over~$\mathbb{T}^2$, we obtain after integration by parts that
\begin{equation}\begin{aligned}\label{88.2}
\frac{d}{dt}\|\nabla\rho\|_{L^q}\leq C\|\nabla \mathbf{u}\|_{L^\infty}\|\nabla\rho\|_{L^r}+C\|\nabla^2\mathbf{u}\|_{L^r}.
\end{aligned}\end{equation}
Similarly, $Z$ also satisfies
\begin{equation}\begin{aligned}\label{1129}
\frac{d}{dt}\|\nabla Z\|_{L^r}\leq& C\|\nabla \mathbf{u}\|_{L^\infty}\|\nabla Z\|_{L^r}+C\|\nabla^2 \mathbf{u}\|_{L^r}.\\
\end{aligned}\end{equation}
In fact, the standard $L^p$- estimate for elliptic system yields
\begin{equation}\begin{aligned}\label{11.07}
\|\nabla^2 \mathbf{u}\|_{L^q}\leq& C\Big(\|\nabla\text{div}\mathbf{u}\|_{L^q}+\|\nabla w\|_{L^q}\Big)\\
\leq&C\left(\|\nabla \Big(\frac{G+P(Z)-\overline{P(Z)}}{2\mu+\lambda}\Big)\|_{L^q}+\|\nabla w\|_{L^q}\right)\\
\leq&C\Big(\|\nabla G\|_{L^q}+\|\nabla w\|_{L^q}+\|\nabla Z\|_{L^r}+\|G\|_{L^\infty}\|\nabla\rho\|_{L^q}+\|\nabla\rho\|_{L^q}\Big)\\
\leq&C\Big(\|\rho\dot{\mathbf{u}}\|_{L^r}+\|\nabla Z\|_{L^q}\Big)+C\Big(\|G\|_{L^\infty}+1\Big)\|\nabla\rho\|_{L^q},\\
\end{aligned}\end{equation}
where we have used \eqref{gw}.
It follows from \eqref{eq:poincare-sobolev}, \eqref{eq:poincare-type}, \eqref{estimateoft} that
\begin{eqnarray}\label{2341}
&&\|\rho\dot{\mathbf{u}}\|_{L^q}\leq C\|\rho\dot{\mathbf{u}}\|^{\frac{2(q-1)}{q^2-2}}_{L^2}\|\rho\dot{\mathbf{u}}\|^{\frac{q(q-2)}{q^2-2}}_{L^{q^2}}\nonumber\\
&\leq&C\|\rho\dot{\mathbf{u}}\|^{\frac{2(q-1)}{q^2-2}}_{L^2}\|\dot{\mathbf{u}}\|^{\frac{q(q-2)}{q^2-2}}_{L^{q^2}}\nonumber\\
&\leq&C\|\rho\dot{\mathbf{u}}\|^{\frac{2(q-1)}{q^2-2}}_{L^2}\|\dot{\mathbf{u}}\|^{\frac{q(q-2)}{q^2-2}}_{H^1}\\
&\leq&C\|\rho\dot{\mathbf{u}}\|^{\frac{2(q-1)}{q^2-2}}_{L^2}\left(\|\sqrt{\rho}\dot{\mathbf{u}}\|^{\frac{q(q-2)}{q^2-2}}_{L^2}+\|\nabla\dot{\mathbf{u}}\|^{\frac{q(q-2)}{q^2-2}}_{L^2}\right)\nonumber\\
&\leq&C\|\rho\dot{\mathbf{u}}\|_{L^2}+C\|\rho\dot{\mathbf{u}}\|^{\frac{2(q-1)}{q^2-2}}_{L^2}\|\nabla\dot{\mathbf{u}}\|^{\frac{q(q-2)}{q^2-2}}_{L^2}\nonumber\\
\end{eqnarray}
Using the inequality above together with \eqref{highestimate}, we have
\begin{equation}\begin{aligned}\label{rhodout}
\int_{0}^{T}&\Big(\|\rho\dot{\mathbf{u}}\|^{1+\frac{1}{q}}_{L^q}+t\|\dot{\mathbf{u}}\|^2_{H^1}\Big)dt\\
\leq&C\int_{0}^{T}\Big(\|\sqrt{\rho}\dot{\mathbf{u}}\|^2_{L^2}+t\|\nabla\dot{\mathbf{u}}\|^2_{L^2}+t^{-1+\frac{2}{p^3-p^2-2p+2}}\Big)dt\\
\leq& C.
\end{aligned}\end{equation}
Next, one gets from the Gagliardo-Nirenberg inequality, \eqref{estimateoft} that
\begin{equation}\begin{aligned}\label{11.16}
&\|\text{div}\mathbf{u}\|_{L^\infty}+\|\omega\|_{L^\infty}+\|G \|_{L^\infty}\\
\leq& C(\|G\|_{L^\infty}+\|P\|_{L^\infty})+C\|\omega\|_{L^\infty}\\
\leq&C\left(\|G\|^{\frac{q-2}{2(q-1)}}_{L^2}\|\nabla G\|^{\frac{q}{2(q-1)}}_{L^q}+\|\omega\|^{\frac{q-2}{2(q-1)}}_{L^2}\|\nabla
w\|^{\frac{q}{2(q-1)}}_{L^q}+1\right)\\
\leq&C\left(1+\|\rho\dot{\mathbf{u}}\|^{\frac{q}{2(q-1)}}_{L^q}\right)\\
\end{aligned}\end{equation}
which together with Lemma \ref{lem:BKM}, \eqref{11.16}, \eqref{11.07} yields that
\begin{equation}\begin{aligned}\label{345}
\|\nabla \mathbf{u}\|_{L^\infty}\leq& C(\|\text{div}\mathbf{u}\|_{L^\infty}+\|\text{curl} \mathbf{u}\|_{L^\infty})\log(e+\|\nabla^2 \mathbf{u}\|_{L^q})+C\|\nabla \mathbf{u}\|_{L^2}+C\\
\leq&
C\left(1+\|\rho\dot{\mathbf{u}}\|^{\frac{q}{2(q-1)}}_{L^q}\right)\log\left(e+\|\rho\dot{\mathbf{u}}\|_{L^q}+\|\nabla
Z\|_{L^q}+\|\nabla
\rho\|_{L^q}\right)+C.\\
\leq&C\Big(1+\|\rho\dot{\mathbf{u}}\|_{L^q}\Big)\log\Big(e+\|\nabla
Z\|_{L^q}+\|\nabla
\rho\|_{L^q}\Big).
\end{aligned}\end{equation}
Set r=q in \eqref{88.2}, \eqref{1129}, we get
\begin{equation}\begin{aligned}\label{1128}
\frac{d}{dt}\Big(\|\nabla\rho\|_{L^q}+\|\nabla Z\|_{L^q}\Big)\leq& C(1+\|\rho\dot{\mathbf{u}}\|_{L^q})\log\Big(e+\|\nabla Z\|_{L^q}+\|\nabla\rho\|_{L^q}\Big)\\
&\Big(\|\nabla\rho\|_{L^{q}}+\|\nabla Z\|_{L^{q}}\Big).
\end{aligned}\end{equation}

Let
$$
f(t)\triangleq e+\|\nabla\rho\|_{L^q}+\|\nabla Z\|_{L^q},\,\,g(t)\triangleq 1+\|\rho\dot{\mathbf{u}}\|_{L^q},
$$
which yields that
$$
(\log f(t))'\leq Cg(t)\log f(t).
$$
Thus, it follows from  Gronwall's inequality \eqref{rhodout} that
\begin{equation}\begin{aligned}\label{390}
\sup_{0\leq t\leq T}(\|\nabla\rho\|_{L^q}+\|\nabla Z\|_{L^q})\leq C,
\end{aligned}\end{equation}
which combining with \eqref{345}, \eqref{390} gives that
\begin{equation}\begin{aligned}\label{391}
&\int_{0}^{T}\|\nabla \mathbf{u}\|_{L^\infty}dt\leq C\int_{0}^{T}(\|\rho\dot{\mathbf{u}}\|_{L^q}+1)dt\leq C.
\end{aligned}\end{equation}
Taking $r=2,$ it thus follows from \eqref{88.2}, \eqref{1129}, \eqref{11.07}, \eqref{391}, that
\begin{equation}\begin{aligned}\label{88.1225}
\sup_{0\leq t\leq T}(\|\nabla\rho\|_{L^2}+\|\nabla Z\|_{L^2})\leq C.
\end{aligned}\end{equation}
According to~\eqref{estimateoft},~\eqref{88.1225}, we get
\begin{equation}
\begin{aligned}\label{nabla2u}
\|\nabla^{2}\mathbf{u}\|_{L^{2}} &\leq C\|\nabla\omega\|_{L^{2}}+C\|\nabla\mathrm{div}\mathbf{u}\|_{L^{2}} \\
&\leq C\|\nabla\omega\|_{L^{2}}+C+C\|\nabla G\|_{L^{2}}+C\|\mathrm{div}\mathbf{u}\|_{L^{2q/(q-2)}}\|\nabla\rho\|_{L^{q}} \\
&\leq C\|\nabla\omega\|_{L^{2}}+C\|\nabla G\|_{L^{2}}+C+C\|\nabla \mathbf{u}\|_{L^{2}}^{(q-2)/q}\|\nabla^{2}\mathbf{u}\|_{L^{2}}^{2/q} \\
&\leq C+\frac{1}{2}\|\nabla^{2}\mathbf{u}\|_{L^{2}}+C\|\rho\dot{\mathbf{u}}\|_{L^{2}},
\end{aligned}
\end{equation}
then 
\begin{equation}
\begin{aligned}
\|\nabla^2\mathbf{u}\|_{L^2}\leq C.
\end{aligned}
\end{equation}
This completes the proof of Lemma \ref{rhohigh}.
\end{proof}
\subsection{Proof of Theorem \ref{th1}}
\begin{prop}\label{prop5.1}
    Assume that \eqref{beta43} holds and that $(\rho_{0}, Z_0, m_{0})$ satisfies \eqref{eq:initial-regularity}. Then there exists a unique strong solution $(\rho, \mathbf{u}, Z)$ to \eqref{trans1}
    in $\Omega \times (0, \infty)$ satisfying \eqref{eq:local-regularity} for any $T \in (0, \infty)$. 
    In addition, for any $q > 2$, $(\rho, Z, \mathbf{u})$ satisfies \eqref{rhozhihg} with some positive constant 
    $C$ depending only on $T, q, \mu, \gamma, \beta$, $\|\mathbf{u}_{0}\|_{H^{1}}$, and $\|\rho_{0}\|_{W^{1,q}},\|Z_{0}\|_{W^{1,q}}$. 
    Moreover, if \eqref{beta23} holds, there exists some positive constant $C$ depending only 
    on $\mu, \beta, \gamma$, $\|\rho_{0}\|_{L^{\infty}},\|Z_{0}\|_{L^{\infty}}$, and $\|\mathbf{u}_{0}\|_{H^{1}}$ such that 
    \eqref{independestimate} and \eqref{nablau} hold.
\end{prop}
\begin{proof}
By Lemma \ref{lem:local-existence}, there exists a local strong solution $(\rho,\mathbf{u}, Z)$ to the system \eqref{trans1} on $\Omega\times(0,T_0]$ for some $T_0>0$ depending on $\inf_{x\in\mathbb{T}^2}(\rho_0(x),Z_0(x))$, and satisfying \eqref{eq:local-regularity} and \eqref{eq:lower-bound}. Define
\begin{equation}\label{eq:5.1}
T^* \triangleq \sup\left\{T \;\bigg|\; \sup_{0\le t\le T}\|(\rho,\mathbf{u},Z)\|_{H^2}<\infty\right\}.
\end{equation}
Clearly, $T^*\ge T_0$. If $T^*<\infty$, we claim that there exists a constant $\hat{C}>0$, possibly depending on $T^*$ and $\inf_{x\in\mathbb{T}^2}(\rho_0(x),Z_0(x))$, such that
\begin{equation}\label{eq:5.2}
\sup_{0\le t\le T}\|(\rho,Z)\|_{H^2} \le \hat{C} \quad \text{for all } 0<T<T^*.
\end{equation}
This, together with \eqref{rhozhihg}, would contradict the definition \eqref{eq:5.1}. Hence we must have $T^*=\infty$.

The estimates \eqref{rhozhihg}, \eqref{nablau} and \eqref{independestimate} follow directly from \eqref{eq:local-regularity}, Lemma \ref{timehighorderestimate} and Lemma \ref{rhohigh}.

It remains to verify \eqref{eq:5.2}. 

By standard arguments combined with Lemma \ref{rhohigh}, we obtain that for any $T \in (0, T^*)$,
\begin{equation}
\begin{aligned}\label{rholow}
\inf_{(x,t)\in\mathbb{T}^2\times(0,T)}\rho(\mathbf{x},t) &\geq \inf_{(x,t)\in\mathbb{T}^2\times(0,T)}\rho_0 \exp\left( -\int_0^T \norm{\div\mathbf{\mathbf{u}}(s)}_{L^\infty} \, ds \right)\\
&\geq \inf_{(x,t)\in\mathbb{T}^2\times(0,T)}\rho_0 \exp(-C T^{*}) \geq \hat{C}^{-1}
\end{aligned}
\end{equation}
for all $(\mathbf{x},t) \in [0, T^{*}) \times \Omega$. Similarly, we have
\begin{equation}
\begin{aligned}\label{zlow}
\inf_{(x,t)\in\mathbb{T}^2\times(0,T)}Z(\mathbf{x},t) \geq \inf_{(x,t)\in\mathbb{T}^2\times(0,T)}Z_0 \exp(-C T^{*}) \geq \hat{C'}^{-1}
\end{aligned}
\end{equation}
for all $(\mathbf{x},t) \in [0, T^{*}) \times \Omega$. 

Because of \eqref{eq:initial-regularity}, we define the initial material derivative as:
\begin{equation}\label{5.5}
\sqrt{\rho} \dot{\mathbf{u}}(x,t=0) = \rho_0^{-1/2} \Big( \mu \Delta \mathbf{u}_0 + \nabla\big((\mu + \lambda(\rho_0))\operatorname{div} \mathbf{u}_0\big) - \nabla P(Z_0) \Big). 
\end{equation}
Integrating \eqref{na} with respect to $t$ over $(0,T)$ and using \eqref{eq:initial-regularity}, \eqref{independestimate}, and \eqref{rholow}, \eqref{zlow}, we obtain
\begin{equation}\label{5.6}
\sup_{0 \le t \le T} \int \rho |\dot{\mathbf{u}}|^2 \, dx + \int_0^T \|\nabla \dot{\mathbf{u}}\|_{L^2}^2 \, dt \le \hat{C}.
\end{equation}
Now, combining \eqref{gwexpression}, \eqref{nabla2u}, \eqref{rhozhihg}, and \eqref{5.5}, we have:
\begin{equation}
\begin{aligned}\label{secondestimate}
& \sup_{0 \le t \le T} \Big( \|\nabla^2 \mathbf{u}\|_{L^2} + \|\nabla G\|_{L^2} + \|\nabla \omega\|_{L^2} \Big) 
+ \int_0^T \Big( \|\nabla^2 G\|_{L^2}^2 + \|\nabla^2 \omega\|_{L^2}^2 \Big) dt \\
&\quad \le \hat{C} \sup_{0 \le t \le T} \|\rho \dot{\mathbf{u}}\|_{L^2} + \hat{C} \int_0^T \|\nabla(\rho \dot{\mathbf{u}})\|_{L^2}^2 \, dt \\
&\quad \le \hat{C} + \hat{C} \int_0^T \Big( \|\nabla \rho\|_{L^4}^2 \|\dot{\mathbf{u}}\|_{L^{2q/(q-2)}}^2 + \|\nabla \dot{u}\|_{L^2}^2 \Big) dt \\
&\quad \le \hat{C} + \hat{C} \int_0^T \|\dot{\mathbf{u}}\|_{H^1}^2 \, dt \le \hat{C}, 
\end{aligned}
\end{equation}
where in the last inequality we have used \eqref{5.6}.

We begin by applying the  operator $\partial_j \partial_i$  to the equation $\eqref{trans1}_3$
\begin{equation}
\begin{aligned}
&\partial_j \partial_i Z_t + \partial_j \partial_i u \cdot \nabla Z + \mathbf{u} \cdot \nabla \partial_j \partial_i Z + \partial_j \mathbf{u} \cdot \nabla \partial_i Z + \partial_i \mathbf{u} \cdot \nabla \partial_j Z \\
&+ \partial_j \partial_i Z \div x(\mathbf{u}) + Z \partial_j \partial_i \div x(\mathbf{u}) + \partial_i Z \partial_j \div x(\mathbf{u}) + \partial_j Z \partial_i \div x(\mathbf{u}) = 0.
\end{aligned}
\end{equation}
Multiplying the resulting equation by 2$\partial_j \partial_i Z$
, integrating over $\Omega$, and applying integration by parts yields
\begin{equation}
\begin{aligned}
\frac{\mathrm{d}}{\mathrm{d}t} \int_{\Omega} |\partial_j \partial_i Z|^2 \,\mathrm{d}x 
&= - \int_{\Omega} |\partial_j \partial_i Z|^2 \, \mathrm{div}_x(\boldsymbol{\mathbf{u}}) \,\mathrm{d}x 
   - 2 \int_{\Omega} \partial_j \partial_i Z \, \partial_i \boldsymbol{\mathbf{u}} \cdot \nabla_x \partial_j Z \,\mathrm{d}x \\
&\quad - 2 \int_{\Omega} \partial_j \partial_i Z \, \partial_j \boldsymbol{\mathbf{u}} \cdot \nabla_x \partial_i Z \,\mathrm{d}x 
   - 2 \int_{\Omega} Z \, \partial_j \partial_i Z \, \partial_j \partial_i \mathrm{div}_x(\boldsymbol{\mathbf{u}}) \,\mathrm{d}x \\
&\quad - 2 \int_{\Omega} \partial_i Z \, \partial_j \partial_i Z \, \partial_j \mathrm{div}_x(\boldsymbol{\mathbf{u}}) \,\mathrm{d}x 
   - 2 \int_{\Omega} \partial_j Z \, \partial_j \partial_i Z \, \partial_i \mathrm{div}_x(\boldsymbol{\mathbf{u}}) \,\mathrm{d}x \\
&\quad - 2 \int_{\Omega} \partial_j \partial_i Z \, \partial_j \partial_i \boldsymbol{\mathbf{u}} \cdot \nabla_x Z \,\mathrm{d}x \\
&\leq \hat{C}\|\nabla \mathbf{u}\|_{L^\infty}\|\nabla^2 Z\|^2_{L^2}+\hat{C}\|\nabla^2 Z\|_{L^2}\||\nabla Z||\nabla^2\mathbf{u}|\|_{L^2}\\
&\quad+\hat{C}\|\nabla^2 Z\|_{L^2}\|Z\|_{L^\infty}\|\nabla^3\mathbf{u}\|_{L^2},
\end{aligned}
\end{equation}
Based on the analogous treatment for the density, we can obtain
\begin{equation}
\begin{aligned}
\frac{d}{dt}\Big(\|\nabla^2\rho\|_{L^2}+\|\nabla^2Z\|_{L^2}\Big)\leq& \hat{C}(1+\|\nabla \mathbf{u}\|_{L^\infty})(\|\nabla^2\rho\|_{L^2}+\|\nabla^2Z\|_{L^2})\\
&+(\||\nabla Z||\nabla^2\mathbf{u}|\|_{L^2}+\||\nabla \rho||\nabla^2\mathbf{u}|\|_{L^2})+C\|\nabla^3\mathbf{u}\|_{L^2}.
\end{aligned}
\end{equation}
An application of the $L^2$-elliptic estimate yields
\begin{equation}
\begin{aligned}
\|\nabla^3 \mathbf{u}\|_{L^2} 
&\leq \hat{C} \|\nabla^2 \operatorname{div} \mathbf{u}\|_{L^2} + \hat{C} \|\nabla^2 \omega\|_{L^2} \\
&\leq \hat{C} \|\nabla^2 ((2\mu + \lambda(\rho)) \operatorname{div} \mathbf{u})\|_{L^2} 
   + \hat{C} \|\nabla \rho \cdot \nabla^2 \mathbf{u}\|_{L^2} 
   + \hat{C} \|\nabla^2 \rho \cdot \nabla \mathbf{u}\|_{L^2} \\
&\quad +  \hat{C} \|\nabla \rho\|_{L^2}^2 \|\nabla \mathbf{u}\|_{L^2} 
   + \hat{C} \|\nabla^2 \omega\|_{L^2} \\
&\leq   \hat{C}\|\nabla^2 G\|_{L^2} +  \hat{C}\|\nabla^2 P(Z)\|_{L^2} +  \hat{C}\|\nabla^2 \omega\|_{L^2} \\
&\quad +  \hat{C}\|\nabla \mathbf{u}\|_{L^\infty} \bigl( \|\nabla^2 \rho\|_{L^2} + \|\nabla \rho\|_{L^4} \bigr) \\
&\quad +  \hat{C}\|\nabla \rho\|_{L^q} \|\nabla^2 \mathbf{u}\|_{L^2}^{1-\frac{2}{q}} \|\nabla^3 \mathbf{u}\|_{L^2}^{\frac{2}{q}}\\
&\leq \hat{C}\|\nabla^2 G\|_{L^2} +  \hat{C}\|\nabla^2 Z\|_{L^2} +  \hat{C}\|\nabla^2 \omega\|_{L^2}+ \hat{C}\|\nabla \mathbf{u}\|_{L^\infty} \bigl( \|\nabla^2 \rho\|_{L^2} + \|\nabla \rho\|_{L^4} \bigr)\\
&\quad +\epsilon\|\nabla^3\mathbf{u}\|_{L^2}
\end{aligned}
\end{equation}
To estimate $\|\nabla^2 P(Z)\|_{L^2}$, we use the following bound:
\begin{equation}
\begin{aligned}
\|\nabla^2 P(Z)\|_{L^2}
&\leq \hat{C} (  \left(  \|p'(Z) \nabla_{\boldsymbol{x}}^{2} Z\|_{L^2} + \|p''(Z) \nabla_{\boldsymbol{x}} Z \otimes \nabla_{\boldsymbol{x}} Z\|_{L^2} \right) \\
&\leq \hat{C}\Big((\gamma-1)Z^{\gamma-1} \| \nabla_{\boldsymbol{x}}^{2} Z\|_{L^2}+\gamma(\gamma-1)Z^{\gamma-2}\|\nabla_{\boldsymbol{x}} Z\|_{L^4}^{2}\Big),
\end{aligned}
\end{equation}
 where the presence of the coefficients $Z^{\gamma-1}$ and $Z^{\gamma-2}$ (with $\gamma > 1$) forces us to establish uniform upper and lower bounds for $Z$.
 
From the estimates above, we get
\begin{equation}
\begin{aligned}
\frac{d}{dt}(\|\nabla^2\rho\|_{L^2}+\|\nabla^2Z\|_{L^2})\leq& \hat{C}(1+\|\nabla u\|_{L^\infty})(\|\nabla^2\rho\|_{L^2}+\|\nabla^2Z\|_{L^2})\\
&+ \hat{C}\|\nabla^2 G\|_{L^2} +  \hat{C}\|\nabla^2 \omega\|_{L^2}.
\end{aligned}
\end{equation}
By Gronwall's inequality and \eqref{secondestimate}, we obtain
\begin{equation}
\begin{aligned}
\sup_{0\leq t\leq T}(\|\nabla^2\rho\|^2_{L^2}+\|\nabla^2Z\|^2_{L^2})\leq \hat{C}.
\end{aligned}
\end{equation}
\end{proof}
\noindent\textbf{Proof of Theorem \ref{th1}}
Let $(\rho_0, m_0,Z_0)$ satisfy the assumption \eqref{ia} in Theorem \ref{th1}, 
For $(\delta,\eta)\in(0,1),$ we define
\begin{equation}
\begin{aligned}\nonumber
\rho_{0}^{\delta,\eta} \triangleq j_{\delta} * \rho_{0} + \eta \geq \eta > 0, \quad \mathbf{u}_{0}^{\delta,\eta} \triangleq j_{\delta} * \mathbf{u}_{0}, \quad Z_{0}^{\delta,\eta} \triangleq j_{\delta} * Z_{0} + A_0^{\delta}\eta\geq A_0^{\delta}\eta > 0,
\end{aligned}
\end{equation}
and 
\begin{equation}
\begin{aligned}\nonumber
m_0^{\delta,\eta}=\rho_{0}^{\delta,\eta}\mathbf{u}_{0}^{\delta,\eta},
\end{aligned}
\end{equation}
where  $j_\delta$ is the standard mollifying of width $\delta.$

By construction, we have $\rho_0^{\delta,\eta}, \mathbf{u}_0^{\delta,\eta},\,Z^{\delta,\eta} \in H^\infty$ and
\begin{equation}
\lim_{\delta+\eta \to 0} \Big( \|\rho_0^{\delta,\eta} - \rho_0\|_{W^{1,q}} + \|\mathbf{u}_0^{\delta,\eta} - \mathbf{u}_0\|_{H^1} + \|Z_0^{\delta,\eta} - Z_0\|_{W^{1,q}}\Big) = 0.
\end{equation}

We now verify that the approximate function $\theta_0$ obtained from the approximations of $\rho_0^{\delta,\eta}$ and $Z_0^{\delta,\eta}$ is indeed admissible. All previous a priori estimates assume $\theta_0$ lies in the fixed interval $[c_*, c^*]$, and the estimates in Section \ref{sec-4} depend explicitly on $c_*$ and $c^*$. We now show that $\theta$ remains uniformly bounded in $[c_{*}, c^{*}]$:
\begin{equation}
\begin{aligned}\label{appro}
 c_*j_{\delta} * \rho_{0}=c_* \int_{\mathbb{R}^3} \rho(t, y) j_{\delta}(x - y) \, dy &\leq j_{\delta} * Z_{0}= \int_{\mathbb{R}^3} Z(t, y) j_{\delta}(x - y) \, dy \\
&\le c^* \int_{\mathbb{R}^3} \rho(t, y) j_{\delta}(x - y) \, dy =c^*j_{\delta} * \rho_{0},
\end{aligned}
\end{equation}
For $\rho_0 > 0$ \begin{equation}
\begin{aligned}
c_{*}=\frac{c_{*}j_{\delta} * \rho_{0} + c_{*}\eta}{j_{\delta} * \rho_{0} + \eta}\leq \frac{Z_{0}^{\delta,\eta}}{\rho_{0}^{\delta,\eta}}=\frac{ j_{\delta} * Z_{0} + A_0^{\delta}\eta}{ j_{\delta} * \rho_{0} + \eta}\leq \frac{c^{*}j_{\delta} * \rho_{0} + c^{*}\eta}{j_{\delta} * \rho_{0} + \eta}= c^*
\end{aligned}
\end{equation}
For $\rho_0 = 0$, we have $Z_0 = 0$ and then we obtain
$$\frac{ j_{\delta} * Z_{0} + A^{\delta}_0\eta}{ j_{\delta} * \rho_{0} + \eta}=A^{\delta}_0.$$ Hence, in the vacuum region, we naturally define $\theta_0 = A_0.$ 

Through the approximate construction described above, the function $\theta_0$ can be properly defined as 
\[
\theta_0 = 
\begin{cases} 
\displaystyle \frac{Z_0}{\rho_0}, & \rho_0 > 0,\\[10pt]
A_0, & \rho_0 = 0,
\end{cases}
\]
where $A_0$ is taken such that it satisfies equation \eqref{A} and the estimate 
$$c_*\leq A_0\leq c^*,$$holds. 
Consequently, we have $c_{*} \leq \theta_0 \leq c^{*}$ almost everywhere in $(0, T) \times \Omega$.

Applying Proposition~\ref{prop5.1} with the data $(\rho_0^{\delta,\eta}, m_0^{\delta,\eta}, Z_0^{\delta,\eta})$ yields a unique global strong solution  
$(\rho^{\delta,\eta}, \mathbf{u}^{\delta,\eta}, Z^{\delta,\eta})$ of \eqref{trans1} on $\Omega \times (0,\infty)$. Moreover, for every $T>0$ the solution satisfies the uniform estimate \eqref{rhozhihg} with a constant $C$ independent of $\delta$ and $\eta$.    
If \eqref{beta23} is satisfied, then the uniform‑in‑time bounds \eqref{independestimate} and \eqref{nablau} are also valid with a constant $C$ depending only on $\mu,\beta,\gamma,\|\rho_0\|_{L^{\infty}},\,\|Z_0\|_{L^{\infty}},\,\|\mathbf{u}_0\|_{H^{1}}$.

Passing successively to the limits $\delta\to0$ and then $\eta\to0$ via standard compactness arguments (see \cite{M. Perepelitsa,VaigantKazhikhov,J. Li}), we obtain a global strong solution $(\rho,u,Z)$ of the original problem.  
This solution satisfies all the regularity properties stated in Theorem~\ref{th1} except for the long-time behavior \eqref{eq:long-time} and the uniqueness statement concerning solutions fulfilling \eqref{theoremregularity}.  
Furthermore, if condition \eqref{beta23} holds, $(\rho,\mathbf{u},Z)$ also satisfies the uniform-in-time bounds \eqref{independestimate} and \eqref{nablau} with a constant $C$ depending only on $\mu,\beta,\gamma,\|\rho_0\|_{L^{\infty}}, \|Z_0\|_{L^{\infty}}$ and $\|\mathbf{u}_0\|_{H^{1}}$.

The uniqueness of solutions satisfying \eqref{theoremregularity} can be proved by a standard energy method; for details we refer to \cite{P. Germain}.  
Similarly, the proof of the asymptotic behavior \eqref{eq:long-time} follows from the arguments presented in \cite{huanglijmpa}.  
Hence the proof of Theorem~\ref{th1} is complete.

\section{Proof of Theorem \ref{th1} on Bounded Domains}\label{sec-5}

The proof of Theorem \ref{th1} on bounded domains can be found in \cite{fanlili,fanliwang}. Here, we only present the parts that differ from \cite{fanlili,fanliwang}; specifically, we provide a proof for the commutator estimates.

In the study of the compressible Navier–Stokes equations in a two-dimensional bounded domain, commutator estimates serve as a crucial tool for establishing global existence of solutions with large initial data. Unlike the approach adopted in \cite{fanlili}, where a Riemann mapping is employed to reduce a general bounded simply-connected domain to the unit disc and a pull-back Green’s function is introduced to obtain a commutator-type representation, the present work fully exploits the structural features of the Navier‑slip boundary condition together with the specific two‑dimensional geometry directly in bounded domains. This enables us to derive effective commutator estimates without relying on the technical complications associated with domain transformations.
\begin{lema}\label{commutator estimatenavierslip}
Let $\Omega\subset\mathbb{R}^{2}$ be a simply connected smooth bounded domain with the Navier-slip boundary condition. For $p, p_{1}, p_{2} \in (1,\infty)$ with
\begin{equation}
\frac{1}{p} = \frac{1}{p_{1}} + \frac{1}{p_{2}}.
\end{equation}
Define $F$ as in \eqref{fdenie}. Then there exists a constant $C = C(\Omega,p,p_1,p_2) > 0$
\begin{equation}
\begin{aligned}\label{commuatotorestiamte1}
\norm{F}_{q} &\leq \norm{\rho}_{L^{\infty}} \norm{\nabla \mathbf{u}}_{L^2}^{2},  \\
\end{aligned}
\end{equation}
and
\begin{equation}
\begin{aligned}\label{commuatotorestiamte2}
\norm{\nabla F}_{p} &\leq \norm{\rho \mathbf{u}}_{L^{p_{1}}} \norm{\nabla \mathbf{u}}_{L^{p_{2}}}.
\end{aligned}
\end{equation}
\end{lema}

\begin{proof}
The proof of \eqref{commuatotorestiamte1} is straightforward and will be omitted.

Define $f_1$ and $f_{2}$ as the unique solutions  of the Neumann problems
\begin{equation}
\begin{aligned}
\begin{cases}
\Delta f_1 = \mathrm{div}\,\mathrm{div} (\rho \mathbf{u} \otimes \mathbf{u}) &\quad \text{in } \Omega,\\[4pt]
\displaystyle \frac{\p f}{\p n} = k\rho|\mathbf{u}|^{2},&\quad \text{on } \partial\Omega
\end{cases}
\qquad
\begin{cases}
\Delta f_{2} = \mathrm{div}(\rho \mathbf{u}) &\quad \text{in } \Omega,\\[4pt]
\displaystyle \frac{\p f_{2}}{\p n} = 0,&\quad \text{on } \partial\Omega
\end{cases}
\end{aligned}
\end{equation}
where the constant $k$ is the curvature of the curve. Since $\Omega$ is a smooth simply connected domain in $\mathbb{R}^2$, 
we can introduce stream functions $\psi_{1}$ and $\psi_{2}$ satisfying
\begin{equation}
\begin{aligned}\label{f2psi}
\nabla f_1 - \mathrm{div} (\rho \mathbf{u} \otimes \mathbf{u}) &= \nabla^{\perp}\psi_{1}, \\
\nabla f_{2} - \rho \mathbf{u} &= \nabla^{\perp}\psi_{2},
\end{aligned}
\end{equation}
with $\nabla^{\perp}\psi_{j}\cdot n = 0$ $(j=1,2)$ on $\p\Omega$. By adding suitable constants we may impose $\psi_{j}|_{\p\Omega}=0$, $(j=1,2)$. Then $\psi_{j}$ solve the Dirichlet problems
\begin{equation}
\begin{aligned}
\begin{cases}
\Delta \psi_{1} = -\nabla^{\perp}\cdot \bigl( \mathrm{div} (\rho \mathbf{u} \otimes \mathbf{u}) \bigr), &\quad \text{in } \Omega,\\[4pt]
\psi_{1}=0,&\quad \text{on } \partial\Omega.
\end{cases}
\qquad
\begin{cases}
\Delta \psi_{2} = -\nabla^{\perp}\cdot (\rho \mathbf{u}), &\quad \text{in } \Omega,\\[4pt]
\psi_{2}=0,&\quad \text{on } \partial\Omega.
\end{cases}
\end{aligned}
\end{equation}
Standard elliptic estimate gives the bounds
\begin{equation}
\begin{aligned}
\norm{\nabla^{\perp}\psi_{1}}_{L^q} &\leq \norm{\mathrm{div} (\rho \mathbf{u} \otimes \mathbf{u})}_{L^q}, \\
\norm{\nabla^{\perp}\psi_{2}}_{L^q} &\leq \norm{\rho \mathbf{u}}_{L^q}.
\end{aligned}
\end{equation}
Since $F = f_1 - \mathbf{u}\cdot\nabla f_{2}$, we compute
\begin{equation}
\begin{aligned}
\nabla F &= \nabla f_1- \nabla (\mathbf{u}\cdot\nabla f_{2}) \\
&= \nabla f_1- \mathbf{u}\cdot\nabla (\nabla f_{2}) + \nabla \mathbf{u}\cdot\nabla f_{2} \\
&= \bigl(\mathrm{div} (\rho \mathbf{u} \otimes \mathbf{u}) + \nabla^{\perp}\psi_{1}\bigr) 
      - \mathbf{u}\cdot\nabla (\rho \mathbf{u} + \nabla^{\perp}\psi_{2}) 
      + \nabla \mathbf{u}\cdot (\rho \mathbf{u} + \nabla^{\perp}\psi_{2}) \\
&= \nabla^{\perp}\psi_{1} - \mathbf{u}\cdot\nabla (\nabla^{\perp}\psi_{2}) 
      + \nabla \mathbf{u}\cdot\nabla^{\perp}\psi_{2} \\
&\mathbf{u}ad + \rho \mathbf{u}\,\mathrm{div} \mathbf{u} + \rho \mathbf{u}\cdot{\nabla \mathbf{u}}.
\end{aligned}
\end{equation}
The terms $\rho \mathbf{u}\,\mathrm{div} \mathbf{u}$, $\nabla \mathbf{u}\cdot\nabla^{\perp}\psi_{2}$ and $\rho \mathbf{u}\cdot{\nabla \mathbf{u}}$ are all of the form whose $L^{p}$-norm is bounded by $\norm{\rho \mathbf{u}}_{L^{p_{1}}}\norm{\nabla \mathbf{u}}_{L^{p_{2}}}$ via H\"older's inequality. We denote their sum by $R^{(1)}$ and note that
\begin{equation}
\norm{R^{(1)}}_{L^p} \leq \norm{\rho \mathbf{u}}_{L^{p_{1}}}\norm{\nabla \mathbf{u}}_{L^{p_{2}}}.
\end{equation}
Thus
\begin{equation}
\nabla F = \nabla^{\perp}\psi_{1} - \mathbf{u}\cdot\nabla (\nabla^{\perp}\psi_{2}) 
          + R^{(1)}.
\end{equation}
Using the identity
\begin{equation}
\mathbf{u}\cdot\nabla (\nabla^{\perp}\psi_{2}) = \nabla^{\perp}(\mathbf{u}\cdot\nabla\psi_{2}) 
                                       - \nabla^{\perp}\mathbf{u}\cdot\nabla\psi_{2},
\end{equation}
we obtain
\begin{equation}
\begin{aligned}
\nabla F &= \nabla^{\perp}\psi_{1} - \nabla^{\perp}(\mathbf{u}\cdot\nabla\psi_{2})
          + \bigl(\nabla^{\perp}\mathbf{u}\cdot\nabla\psi_{2} 
                \bigr) + R^{(1)}.
\end{aligned}
\end{equation}
The combination $\nabla^{\perp}\mathbf{u}\cdot\nabla\psi_{2}$ is again bounded by $\norm{\rho \mathbf{u}}_{L^{p_{1}}}\norm{\nabla \mathbf{u}}_{L^{p_{2}}}$.  Denoting this term together with $R^{(1)}$ as $R^{(2)}$, we have
\begin{equation}
\begin{aligned}
\nabla F &= \nabla^{\perp}\bigl( \psi_{1} - \mathbf{u}\cdot\nabla\psi_{2} \bigr) + R^{(2)}, \\
\end{aligned}
\end{equation}
where we have 
\begin{equation}
\begin{aligned}
\norm{R^{(2)}}_{p} &\leq \norm{\rho \mathbf{u}}_{L^{p_{1}}}\norm{\nabla \mathbf{u}}_{L^{p_{2}}}.
\end{aligned}
\end{equation}
Set $g := \psi_{1} - \mathbf{u}\cdot\nabla\psi_{2}$. Then $g|_{\p\Omega}=0$ and it suffices to estimate $\nabla^{\perp}g$.
Applying the Laplacian gives
\begin{equation}
\begin{aligned}
\Delta g &= \Delta\psi_{1} - \Delta(\mathbf{u}\cdot\nabla\psi_{2}) \\
&= -\nabla^{\perp}\cdot\bigl( \mathrm{div} (\rho \mathbf{u} \otimes \mathbf{u}) \bigr) 
     - \nabla^{\perp}\cdot\nabla^{\perp}(\mathbf{u}\cdot\nabla\psi_{2}) \\
&= -\nabla^{\perp}\cdot(\rho \mathbf{u}\,\mathrm{div} \mathbf{u}) 
     - \nabla^{\perp}\cdot\bigl( \mathbf{u}\cdot\nabla (\rho \mathbf{u}) \bigr) \\
&\quad - \nabla^{\perp}\cdot\bigl( \nabla^{\perp}\mathbf{u}\cdot\nabla\psi_{2}
                               + \mathbf{u}\cdot\nabla\nabla^{\perp}\psi_{2} \bigr).
\end{aligned}
\end{equation}
We define $\nabla^{\perp}\cdot(\rho \mathbf{u}\,\mathrm{div} \mathbf{u})$ and $\nabla^{\perp}\cdot(\nabla^{\perp}\mathbf{u}\cdot\nabla\psi_{2})$ as $R^{(3)}.$ The term 
$R^{(3)}$
satisfies the estimate 
\begin{equation}
\begin{aligned}
\norm{\nabla \Delta^{-1}_0 R^{(3)}}_{L^p} &\leq \norm{\rho \mathbf{u}}_{L^{p_{1}}}\norm{\nabla \mathbf{u}}_{L^{p_{2}}}.
\end{aligned}
\end{equation}
Hence we may write
\begin{equation}
\begin{aligned}\label{deltag}
\Delta g &= -\nabla^{\perp}\cdot\bigl( \mathbf{u}\cdot\nabla (\rho \mathbf{u}) \bigr)
         - \nabla^{\perp}\cdot\bigl( \mathbf{u}\cdot\nabla\nabla^{\perp}\psi_{2} \bigr) + R^{(3)}, 
\end{aligned}
\end{equation}
where

\begin{equation}
\begin{aligned}
\norm{\nabla \Delta^{-1}_0R^{(3)}}_{L^p} &\leq \norm{\rho \mathbf{u}}_{L^{p_{1}}}\norm{\nabla \mathbf{u}}_{L^{p_{2}}}.
\end{aligned}
\end{equation}
Now we compute the second term of the above equation
\begin{equation}
\begin{aligned}
\nabla^{\perp}\cdot\bigl( \mathbf{u}\cdot\nabla\nabla^{\perp}\psi_{2} \bigr)
   &= \mathbf{u}\cdot\nabla (\nabla^{\perp}\cdot\nabla^{\perp}\psi_{2})
     + \nabla^{\perp}\mathbf{u} : \nabla\nabla^{\perp}\psi_{2}.
\end{aligned}
\end{equation}
Combining with \eqref{f2psi}$_2$, we have
\begin{equation}
\begin{aligned}\label{sec5}
\nabla^{\perp}\cdot\bigl( \mathbf{u}\cdot\nabla\nabla^{\perp}\psi_{2} \bigr)
   &= -\mathbf{u}\cdot\nabla\bigl( \nabla^{\perp}\cdot(\rho \mathbf{u}) \bigr)
     + \nabla^{\perp}\mathbf{u} : \nabla\nabla^{\perp}\psi_{2}.
\end{aligned}
\end{equation}
Substituting \eqref{sec5} into \eqref{deltag} yields
\begin{equation}
\begin{aligned}
\Delta g &= -\nabla^{\perp}\cdot\bigl( \mathbf{u}\cdot\nabla (\rho \mathbf{u}) \bigr)
          + \mathbf{u}\cdot\nabla\bigl( \nabla^{\perp}\cdot(\rho \mathbf{u}) \bigr)
          - \nabla^{\perp}\mathbf{u} : \nabla\nabla^{\perp}\psi_{2} + R^{(3)} \\
      &= -\nabla^{\perp}\mathbf{u} : \nabla (\rho \mathbf{u})
          - \nabla^{\perp}\mathbf{u} : \nabla\nabla^{\perp}\psi_{2} + R^{(3)} \\
      &= -\nabla^{\perp}\mathbf{u} : \nabla\bigl( \rho \mathbf{u} + \nabla^{\perp}\psi_{2} \bigr) + R^{(3)}.
\end{aligned}
\end{equation}
Observe that
\begin{equation}
\begin{aligned}
\nabla^{\perp}\mathbf{u} : \nabla\bigl( \rho \mathbf{u} + \nabla^{\perp}\psi_{2} \bigr)
   &= \mathrm{div}\bigl( (\rho \mathbf{u} + \nabla^{\perp}\psi_{2})\cdot\nabla^{\perp}\mathbf{u} \bigr) \\
   &\quad - (\rho \mathbf{u} + \nabla^{\perp}\psi_{2})\cdot\nabla^{\perp}\,\mathrm{div} \mathbf{u}.
\end{aligned}
\end{equation}
The first term is a total divergence. For the second term we compute
\begin{equation}
\begin{aligned}
(\rho \mathbf{u} + \nabla^{\perp}\psi_{2})\cdot\nabla^{\perp}\,\mathrm{div} \mathbf{u}
   &= \rho \mathbf{u}\cdot\nabla^{\perp}\,\mathrm{div} \mathbf{u} 
      + \nabla^{\perp}\psi_{2}\cdot\nabla^{\perp}\,\mathrm{div} \mathbf{u} \\
   &= \rho \mathbf{u}\cdot\nabla^{\perp}\,\mathrm{div} \mathbf{u} 
      + \nabla^{\perp}\cdot(\nabla^{\perp}\psi_{2}\,\mathrm{div} \mathbf{u}) 
      - (\nabla^{\perp}\cdot\nabla^{\perp}\psi_{2})\mathrm{div} \mathbf{u} \\
   &= \rho \mathbf{u}\cdot\nabla^{\perp}\,\mathrm{div} \mathbf{u} 
      + \nabla^{\perp}\cdot(\nabla^{\perp}\psi_{2}\,\mathrm{div} \mathbf{u}) 
      + \nabla^{\perp}\cdot(\rho \mathbf{u}\,\mathrm{div} \mathbf{u})\\
     &= \nabla^{\perp}\cdot(\rho \mathbf{u}\,\mathrm{div} \mathbf{u})+\nabla^{\perp}\cdot(\nabla^{\perp}\psi_{2}\,\mathrm{div} \mathbf{u}) 
\end{aligned}
\end{equation}
These two terms admit the same estimate as $R^{(3)}$	
 . We combine them with 
$R^{(3)}$ to define a new remainder term $R^{(4)}.$	
 
 Collecting divergence terms into a single term $\mathrm{div} V$ with
\begin{equation}
\norm{V}_{L^p} \leq \norm{\rho \mathbf{u}}_{L^{p_{1}}}\norm{\nabla \mathbf{u}}_{L^{p_{2}}},
\end{equation}
we conclude that
\begin{equation}
\begin{aligned}
\begin{cases}
\Delta g = \mathrm{div} V + R^{(4)}, &\quad \text{in } \Omega,\\
g=0,&\quad \text{in } \partial\Omega,
\end{cases}
\end{aligned}
\end{equation}
where $$\norm{\nabla \Delta^{-1}_0R^{(3)}}_{L^p}\leq \norm{\rho \mathbf{u}}_{L^{p_{1}}}\norm{\nabla \mathbf{u}}_{L^{p_{2}}}.$$
So the elliptic estimate gives
\begin{equation}
\begin{aligned}
\norm{\nabla g}_{L^p} 
   &\leq \norm{V}_{L^p} + \norm{\nabla\Delta^{-1}_0R^{(4)}}_{L^p} \\
   &\leq \norm{\rho \mathbf{u}}_{L^{p_{1}}}\norm{\nabla \mathbf{u}}_{L^{p_{2}}}.
\end{aligned}
\end{equation}
Finally, recalling that
\begin{equation}
\begin{aligned}
\nabla F &= \nabla^{\perp}g + R^{(2)}, \\
\norm{R^{(2)}}_{L^p} &\leq \norm{\rho \mathbf{u}}_{L^{p_{1}}}\norm{\nabla u}_{L^{p_{2}}},
\end{aligned}
\end{equation}
we obtain the desired bound
\begin{equation}
\norm{\nabla F}_{L^p} \leq \norm{\rho \mathbf{u}}_{L^{p_{1}}}\norm{\nabla \mathbf{u}}_{L^{p_{2}}}.
\end{equation}
Thus the proof of the Lemma is completed.
\end{proof}

\section*{Data availability}
No data was used for the research described in the article.
\section*{ Conflict of interest}
The authors declare that they have no conflict of interest.

\section*{Ackonwledgments}
Xiangdi Huang is partially supported by Chinese Academy of Sciences Project for Young Scientists in Basic Research (Grant No. YSBR-031), National Natural Science Foundation of
China (Grant Nos. 12494542, 11688101). Jie Fan was supported by the China Postdoctoral Science Foundation (Grant 2025M783153). The author wishes to acknowledge the helpful comments and suggestions provided by Dr. Xinyu Fan, which have greatly enhanced the quality of this paper.

\end{document}